\documentclass[a4paper,12pt]{amsart}
\usepackage{mathrsfs}
\usepackage{amsfonts,amssymb,amscd,amsthm,amsmath,graphicx}
\usepackage{longtable}
\usepackage[all]{xy}

\theoremstyle{plain}
\newtheorem{theorem}{Theorem}[section]
\newtheorem{lemma}[theorem]{Lemma}
\newtheorem{corollary}[theorem]{Corollary}
\newtheorem{proposition}[theorem]{Proposition}
\newtheorem{claim}{Claim}[section]
\newtheorem{definition}{Definition}[section]
\newtheorem{remark}{Remark}[section]

\newtheorem{example}{Example}[section]

\newtheorem{question}{Question}[section]

\def\A{\operatorname{A}}
\def\B{\operatorname{B}}
\def\C{\operatorname{C}}
\def\BC{\operatorname{BC}}
\def\D{\operatorname{D}}
\def\E{\operatorname{E}}
\def\F{\operatorname{F}}
\def\G{\operatorname{G}}
\def\GL{\operatorname{GL}}

\def\SO{\operatorname{SO}}
\def\Sp{\operatorname{Sp}}
\def\SU{\operatorname{SU}}
\def\U{\operatorname{U}}

\def\Ad{\operatorname{Ad}}

\def\Aut{\operatorname{Aut}}

\def\deg{\operatorname{deg}}

\def\det{\operatorname{det}}

\def\diag{\operatorname{diag}}

\def\Fix{\operatorname{Fix}}

\def\Hom{\operatorname{Hom}}
\def\id{\operatorname{id}}

\def\Im{\operatorname{Im}}
\def\Ind{\operatorname{Ind}}
\def\Int{\operatorname{Int}}

\def\ker{\operatorname{ker}}

\def\Lie{\operatorname{Lie}}
\def\max{\operatorname{max}}

\def\sgn{\operatorname{sgn}}

\def\span{\operatorname{span}}

\newcommand{\bbZ}{\mathbb{Z}}

\newcommand{\bbQ}{\mathbb{Q}}

\newcommand{\frg}{\mathfrak{g}}
\newcommand{\frh}{\mathfrak{h}}

\newcommand{\frs}{\mathfrak{s}}

\newcommand{\fru}{\mathfrak{u}}

\begin{document}

\title[On the dimension datum of a subgroup. II.]{On the dimension datum of a subgroup. II.} 

\date{}

\author{Jun Yu}
\address{BICMR, Peking University, No. 5 Yiheyuan Road, Haidian District, Beijing 100871, China.}
\email{junyu@bicmr.pku.edu.cn}

\abstract{This paper studies three aspects around dimension datum: (1), a generalization of the dimension 
datum, which we call the {\it $\tau$-dimension datum}; (2), dimension data of disconnected subgroups; 
(3), compactness of isospectral sets of normal homogeneous spaces. In the first aspect, we find interesting 
examples of pairs $(H_{i},\tau_{i})$ ($i=1,2$) of subgroups and irreducible representations with the same 
$\tau$-dimension datum, and give an application in constructing isospectral hermitian vector bundles. 
We show that $\tau$-dimension data for all linear characters together determine the image of the 
homomorphism from a connected subgroup up to isomorphism. In the second aspect, we express dimension datum 
in terms of {\it characters} supported on {\it maximal commutative connected subsets}, and give formulas 
for these characters using data associated to {\it affine root systems}. In the last aspect, we show that 
any compact semisimple Lie group has only finitely many possible normal homogeneous quotients up to 
diffeomorphism. We also list some open questions about dimension datum and $\tau$-dimension datum.} 

\endabstract


\thanks{}


\maketitle

\tableofcontents

\section{Introduction}

Dimension datum is first seriously studied by Larsen and Pink (\cite{Larsen-Pink}), motivated by its possible 
application in arithmetic geometry (\cite{Katz}, \cite{Larsen-Pink2}). A striking result they showed is that 
any semisimple closed subgroup is determined by its dimension datum up to isomorphism. Recently dimension datum 
catches more attention due to its critical role in Langlands' program of beyond endoscopy (cf. \cite{Langlands} and 
\cite{Arthur}). Motivated by this, we started a systematic study of dimension data of connected closed subgroups 
which are not necessarily semisimple in \cite{An-Yu-Yu}. A particular important finding in \cite{An-Yu-Yu} is a 
family of non-isomorphic connected closed subgroups which have the same dimension datum. This is a new feature 
for dimension data outside the semisimple subgroups case, and it has an interesting application in constructing 
isospectral manifolds. Further in \cite{Yu-dimension}, we classified connected closed subgroups with the same 
dimension datum, and we also classified connected closed subgroups with linearly dimension data in some sense. 
This could be viewed as a nearly complete understanding for dimension data of connected closed subgroups in a 
compact Lie group. In another aspect, topological properties of dimension datum (finiteness, rigidity, Coxeter 
number) are studied in \cite{An-Yu-Yu} and \cite{Larsen}, and stronger results with different proofs are shown 
in \cite{Yu-compactness}.  

The goal of this paper is to extend the study of dimension datum in three aspects. The first is to study a more 
general notion of $\tau$-dimension datum. We show interesting examples of different pairs $(H_{i},\tau_{i})$ 
with the same $\tau$-dimension datum, and use them to construct isospectral hermitian vector bundles. We also 
show that $\tau$-dimension data for all linear characters suffices to determine the image of the homomorphism 
from a connected compact Lie group up to isomorphism. The second is to study dimension data of disconnected 
subgroups. For this we set up a strategy by studying characters associated to affine root systems. The third 
is to further study compactness of isospectral sets of normal homogeneous spaces. In every aspect the study 
is not complete yet, major remaining questions are asked in Questions \ref{Q:tau-dimension}, \ref{Q:FR}, 
\ref{Q:finiteness-normal}. Moreover, in the last section, we propose some other open questions about 
dimension datum and $\tau$-dimension datum. 

\smallskip

\noindent{\it Notation and conventions.} For a compact Lie group $G$, write $G^{0}$ for the neutral subgroup 
containing $e\in G$; write $\mathfrak{g}_{0}$ for the Lie algebra of $G$; write $\mathfrak{g}=
\mathfrak{g}_{0}\otimes_{\mathbb{R}}\mathbb{C}$ for the complefixed Lie algebra of $G$. 

For a connected compact Lie group $G$, write $G^{s}$ (or $G_{der}$) for the derived subgroup, i.e., 
$$G^{s}=G_{der}=[G,G].$$ 

For a compact abelian group $S$, write $$X^{\ast}(S)=\Hom(S,\U(1))$$ for the character group of $S$. In case 
$S$ is connected, i.e., it is a torus, write $$\Lambda_{S}=X^{\ast}(S).$$ It is also called the {\it weight 
lattice} of $S$. Write $$X_{\ast}(S)=\Hom(\U(1),S)$$ for the cocharacter group of a torus $S$. 


\section{$\tau$-dimension datum} 


Let $G$ be a compact Lie group. For a closed subgroup $H$ of $G$ and an irreducible finite-dimensional 
complex linear representation $\tau$ of $H$, let $\mathscr{D}_{H,\tau}$ denote the map from $\widehat{G}$ 
to $\bbZ$, $$\mathscr{D}_{H,\tau}: \rho\mapsto\dim\Hom_{H}(\tau,\rho|_{H}),$$ where $\widehat{G}$ is the set 
of equivalence classes of irreducible finite-dimensional complex linear representations of $G$ and 
$\Hom_{H}(\tau,\rho|_{H})$ is the space of $H$-equivariant linear maps from $\tau$ to $\rho|_{H}$. We call 
$\mathscr{D}_{H,\tau}$ the {\it $\tau-$dimension datum} of $H$. When $\tau=1$, $\mathscr{D}_{H,\tau}=
\mathscr{D}_{H}$ is the {\it dimension datum} of $H$ (\cite{Larsen-Pink}, \cite{An-Yu-Yu}, \cite{Yu-dimension}). 
In \cite{An-Yu-Yu} and \cite{Yu-dimension}, we have constructed non-isomorphic connected closed subgroups 
with the same dimension datum, e.g, two closed subgroups $H_1,H_2$ in $G=\SU(4n+1)$ with $H_1\cong\U(2n+1)$ 
and $H_2\cong\Sp(n)\times\SO(2n+2)$. In \cite{Yu-dimension}, we made a classification of connected closed 
subgroups with the same dimension datum or with linearly-dependent dimension data. The dimension datum 
problem (\cite{Yu-dimension}) asks: to what extent is \(H\) (up to \(G\)-conjugacy) determined by its 
dimension datum \(\mathscr{D}_H\)? Analogously, we could ask the following question. 

\begin{question}\label{Q:tau-dimension datum}
To what extent is a pair $(H,\tau)$ (up to \(G\)-conjugacy) determined by the $\tau$-dimension datum 
\(\mathscr{D}_{H,\tau}\)? 
\end{question}

It is also interesting to construct non-isomorphic connected closed subgroups $H_1,H_2$ in a compact Lie 
group $G$ and non-trivial irreducible representations $\tau_1\in\widehat{H_1}$ and $\tau_2\in\widehat{H_2}$ 
such that $\mathscr{D}_{H_1,\tau_1}=\mathscr{D}_{H_2,\tau_2}$, and to study linear relations among 
$\tau$-dimension data.  

\smallskip 

In this section we present a family of tuples $(G,H_1,\tau_1,H_2,\tau_2)$ ($\tau_1\neq 1$, $\tau_2\neq 1$) 
such that $\mathscr{D}_{H_1,\tau_1}=\mathscr{D}_{H_2,\tau_2}$, and we give a short remark for what we know 
about Question \ref{Q:tau-dimension datum}. This example has an application in constructing isospectral 
hermitian vector bundles. In the last subsection, we show that the $\tau$-dimension data for all linear 
together determine the image up to isomorphism for a homomorpism from a given connected compact Lie group 
to a connected compact Lie group. 

\subsection{An interesting example}

In $G=\SU(4n+2)$, set $$H_1=\{A,\overline{A}: A\in\U(2n+1)\},$$ $$H_2=\{(A,B): A\in\Sp(2n), B\in\SO(2n+2)\}.$$ 
Then, $H_1\cong\U(2n+1)$, $H_2\cong\Sp(n)\times\SO(2n+2)$. For a sequence of integers $a_1\geq a_2\geq\cdots
\geq a_{2n+1}$ with $a_{i}+a_{2n+2-i}=0$ for any $i$, $1\leq i\leq n$, write $\lambda=(a_1,a_2,\dots,a_{2n+1})$ 
for a weight of $H_1\cong\U(2n+1)$. Write $\lambda_1=(a_1,\dots,a_{n})$ for a weight of $\Sp(2n)$, $\lambda_2=
(a_{1},\dots,a_{n+1})$ for a weight of $\SO(2n+2)$, and $\lambda'=(\lambda_1,\lambda_2)$ for a weight of $H_2$. 
Write $\tau_{\lambda}$ ($\tau_{\lambda'}$) for an irreducible representation of $H_1$ (of $H_2$) with highest 
weight $\lambda$ ($\lambda'$). 

\begin{theorem}\label{T:tau-example}
In the above setting, $\mathscr{D}_{H_1,\tau_{\lambda}}=\mathscr{D}_{H_2,\tau_{\lambda'}}$. 
\end{theorem}

Note that, when $a_1=a_2=\cdots=a_{2n+1}=0$, Theorem \ref{T:tau-example} is just \cite[Theorem 1.5(1)]{An-Yu-Yu}. 

\medskip 

The proof of Theorem \ref{T:tau-example} consists of three steps. 

\noindent{\it Step 1, root systems, characters, and dimension datum.} 

Let $T$ be a closed torus in $G$. Write $\Lambda_{T}$ for the weight lattice of $T$. Write $$\Gamma^{\circ}
=N_{G}(T)/Z_{G}(T).$$ Choose a biinvariant Riemannian metric on $G$. Restricting to $T$ it gives a positive 
definite inner product on the Lie algebra of $T$, hence also induces a positive definite inner product on 
$\Lambda_{T}$, both are $\Gamma^{\circ}$ invariant. Let $(\cdot,\cdot)$ denote the induced inner product on 
$\Lambda_{T}$. As in \cite[Definition 2.2]{Yu-dimension}, a root system in $\Lambda_{T}$ is a finite subset 
$\Phi$ of $\Lambda_{T}$ satisfying the following conditions,  
\begin{enumerate}
\item For any two roots $\alpha\in\Phi$ and $\beta\in\Phi$, the element $\beta-\frac{2(\beta,\alpha)}
{(\alpha,\alpha)}\alpha\in\Phi$.
\item (\textbf{Strong integrality}) For any root $\alpha$ and any vector $\lambda\in\Lambda_{T}$, the 
number $\frac{2(\lambda,\alpha)}{(\alpha,\alpha)}$ is an integer.
\end{enumerate}

Let $\Psi'_{T}$ be the root system in $\Lambda_{T}$ generated by root systems $\Phi(H,T)$ where $H$ run 
over connected closed subgroups of $G$ with $T$ a maximal torus of $H$. It is clear that $\Psi'_{T}$ is 
$\Gamma^{\circ}$ stable. By \cite[Corollary 3.4]{Yu-dimension}, $\Psi'_{T}$ is the union of root systems 
$\Phi(H,T)$ where $H$ run over connected closed subgroups of $G$ with $T$ a maximal torus of $H$. Hence, 
$W_{\Psi'_{T}}\subset\Gamma^{\circ}$. Choose (and fix) a system of positive roots $\Psi_{T}^{'+}$. 

Analogous to \cite[Definition 3.5]{Yu-dimension}, we introduce some characters. For a root system $\Phi$ 
in $\Lambda_{T}$, set \[\delta_{\Phi}=\frac{1}{2}\sum_{\alpha\in\Phi\cap\Psi_{T}^{'+}}\alpha.\] For a 
root system $\Phi$ in $\Lambda_{T}$ and a weight $\lambda\in\Lambda_{T}$ which is dominant and integral 
for $\Phi$, set \[A_{\Phi,\lambda}=\sum_{w\in  W_{\Phi}}\sgn(w)[\lambda+\delta_{\Phi}-w\delta_{\Phi}]
\in\bbQ[\Lambda_{T}].\] For a finite group $W$ between $W_{\Phi}$ and $\Gamma^{\circ}$, set
\[F_{\Phi,\lambda,W}=\frac{1}{|W|}\sum_{\gamma\in W}\gamma(A_{\Phi,\lambda})\in\bbQ[\Lambda_{T}].\]
For a weight $\lambda\in\Lambda_{T}$ and a finite subgroup $W$ of $\Gamma^{\circ}$, set
\[\chi^{\ast}_{\lambda,W}=\frac{1}{|W|}\sum_{\gamma\in W}[\gamma\lambda]\in\bbQ[\Lambda_{T}].\]

\begin{proposition}\label{P:dimension-character}
If $\mathscr{D}_{H_1,\tau_1}=\mathscr{D}_{H_2,\tau_2}$ for two connected closed subgroups $H_1,H_2$ and 
irreducible representations $\tau_1\in\widehat{H_1}$ and $\tau_{2}\in\widehat{H_2}$, then $H_1$ and $H_2$ 
have conjugate maximal tori. 

Assume that $T$ is a maximal torus of both $H_1$ and $H_2$, write $\Phi_{i}\subset\Lambda_{T}$ for the 
root system of $H_{i}$ ($i=1,2$). Then, $\mathscr{D}_{H_1,\tau_1}=\mathscr{D}_{H_2,\tau_2}$ if and only 
if $$F_{\Phi_1,\lambda_1,\Gamma^{\circ}}=F_{\Phi_2,\lambda_2,\Gamma^{\circ}},$$ where $\lambda_{i}\in
\Lambda_{T}$ is highest weight of $\tau_{i}$ ($i=1,2$). 
\end{proposition}

\begin{proof}
This is analogous to Propositions 3.7 and 3.8 in \cite{Yu-dimension}. Based on the proof given there, 
the only new adding is calalation for $F_{\Phi}(t)\chi_{\lambda}(t)$, where $H$ is a connected closed 
subgroup of $G$ with $T$ a maximal torus of $H$, $\Phi\subset\Lambda_{T}$ is root system of $H$, 
$F_{\Phi}$ is Weyl product of $H$, and $\chi_{\lambda}$ is the character of an irreducible representation 
of $H$ with highest wight $\lambda$. The calculation goes as follows, \begin{eqnarray*}&& |W_{\Phi}|F_{\Phi}(t)
\chi_{\lambda}(t)=\chi_{\lambda}\prod_{\alpha\in\Phi}\big(1-[\alpha]\big)\\&=&\prod_{\alpha\in\Phi^{+}}
\big([\frac{-\alpha}{2}]-[\frac{\alpha}{2}]\big)(\chi_{\lambda}\prod_{\alpha\in\Phi^{+}}\big([\frac{\alpha}{2}]
-[\frac{-\alpha}{2}]\big))\\&=&\big(\sum_{w\in W_{\Phi}}\sgn(w)[-w\delta]\big)\big(\sum_{\tau\in W_{\Phi}}
\sgn(\tau)[\tau(\lambda+\delta)]\big)\\&=&\sum_{w,\tau\in W}\sgn(w)\sgn(\tau)[-w\delta+\tau(\lambda+\delta)]
\\&=_{w\mapsto\tau w}&\sum_{\tau\in W_{\Phi}}\tau\big(\sum_{w\in W_{\Phi}}\sgn(w)[\lambda+\delta-w\delta]\big)
\\&=&|W_{\Phi}|F_{\Phi,\lambda,W_{\Phi}}.\end{eqnarray*} In the above $\delta=\delta_{\Phi}$. Due to $W_{\Phi}
\subset\Gamma^{0}$, we also note that $$\frac{1}{|\Gamma^{0}|}\sum_{\gamma\in\Gamma^{0}}\gamma\cdot 
F_{\Phi,\lambda,W_{\Phi}}=F_{\Phi,\lambda,\Gamma^{\circ}}.$$
\end{proof}

In the situation of Theorem \ref{T:tau-example}, write $$T=\{\diag\{a_1,\dots,a_{2n+1},a_1^{-1},\dots,
a_{2n+1}^{-1}\}: |a_1|=\cdots=|a_{2n+1}|=1\}.$$ Then, $T$ is conjugate to maximal tori of $H_1$ and $H_2$. By 
calculation we see that $$\Psi'_{T}=\BC_{2n+1}$$ and $$\Gamma^{\circ}=W_{\BC_{2n+1}}=\{\pm{1}\}^{2n+1}\rtimes 
S_{2n+1}.$$ Substituting $H_2$ by a conjugate subgroup $H'_2$, we may assume that $T$ is a maximal torus of 
both $H_1$ and $H'_2$, and idendity the root system of $H_1$ (of $H'_2$) with the sub-root system $\A_{2n}$ 
($\C_{n}\sqcup\D_{n+1}$) of $\Psi'_{T}=\BC_{2n+1}$. By Proposition \ref{P:dimension-character}, the conclusion 
of Theorem \ref{T:tau-example} is equivalent to \begin{equation}\label{Eq:A=CD}F_{\A_{2n},\lambda,W_{\BC_{2n+1}}}
=F_{\C_{n}\sqcup\D_{n+1},\lambda',W_{\BC_{2n+1}}}.\end{equation}

\medskip 
 
\noindent{\it Step 2, sub-root systems of $\Psi=\BC_{n}$, polynomials $a_{n}(\lambda)$, $b_{n}(\lambda)$, 
$c_{n}(\lambda)$, $d_{n}(\lambda)$ and their multiplicative relations.}  

There is a good idea in \cite{Larsen-Pink} which transfers characters $F_{\Phi,0,W_{\BC_{n}}}$ into 
polynomials. In \cite{An-Yu-Yu} and \cite{Yu-dimension}, we further find matrix expression for the resulting 
polynimials. Here, we extend these to the characters $F_{\Phi,\lambda,W_{\BC_{n}}}$. 

Following \cite[Section 7]{Yu-dimension}, we briefly recall the idea of \cite{Larsen-Pink} which identifies 
the direct limit of character groups with polynomial ring. Set \begin{eqnarray*}&& \bbZ^{n}:=\bbZ\BC_{n}=
\Lambda_{\BC_{n}}=\span_{\bbZ}\{e_1,e_2,...,e_{n}\},\\&& W_{n}:=\Aut(\BC_{n})=W_{\BC_{n}}=
\{\pm{1}\}^{n}\rtimes S_{n},\\&&\bbZ_{n}:=\bbQ[\bbZ^{n}],\\&& Y_{n}:=\bbZ_{n}^{W_{n}}. \end{eqnarray*} For 
$m\leq n$, the injection\[\bbZ^{m}\hookrightarrow\bbZ^{n}:(a_1,...,a_{m})\mapsto(a_1,...,,a_{m},0,...,0)\] 
extends to an injection $i_{m,n}:\bbZ_{m}\hookrightarrow\bbZ_{n}$. Define $\phi_{m,n}:\bbZ_{m}\rightarrow
\bbZ_{n}$ by \[\phi_{m,n}(z)=\frac{1}{|W_{n}|}\sum_{w\in W_{n}}w(i_{m,n}(z)).\] Thus $\phi_{m,n}\phi_{k,m}=
\phi_{k,n}$ for any $k\leq m\leq n$ and the image of $\phi_{m,n}$ lies in $Y_n$. Hence $\{Y_{m}:\phi_{m,n}\}$ 
forms a direct system and we define \[Y=\lim_{\longrightarrow_{n}} Y_{n}.\] Define the map $j_{n}:\bbZ_{n}
\rightarrow Y$ by composing $\phi_{n,p}$ with the injection $Y_{p}\hookrightarrow Y$. The isomorphism
$\bbZ^{m}\oplus\bbZ^{n}\longrightarrow\bbZ^{m+n}$ gives a canonical isomorphism $M: \bbZ_{m}\otimes_{\bbQ}
\bbZ_{n}\longrightarrow\bbZ_{m+n}$. Given two elements of $Y$ represented by $y\in Y_{m}$ and $y'\in Y_{n}$ 
we define \[yy'=j_{m+n}(M(y\otimes y')).\] This product is independent of the choice of $m$ and $n$ and 
makes $Y$ a commutative associative algebra.

The monomials $[e_1]^{k_1}\cdots[e_{n}]^{k_{n}}$ ($k_1,k_2,\cdots,k_{n}\in\bbZ$) form a $\bbQ$ basis of
$\bbZ_{n}$, where $[e_i]^{k_i}=[k_ie_i]\in\bbZ_1$ is a linear character. Hence $Y$ has a $\bbQ$ basis
\[e(k_1,k_2,...,k_{n})=j_{n}([e_1]^{k_1}\cdots[e_{n}]^{k_{n}})\] indexed by $n\geq 0$ and
$k_1\geq k_2\geq\cdots\geq k_{n}\geq 0$. Mapping $e(k_1,k_2,...,k_{n})$ to $x_{k_1}x_{k_2}\cdots x_{k_{n}}$,
we get a $\bbQ$ linear map \[E: Y\longrightarrow \bbQ[x_0,x_1,...,x_{n},...].\] This map $E$ is an algebra
isomorphism. Here $x_0=1$ and write $x_{0}$ for notational convenience. For any $k_1\geq k_2\geq\cdots
\geq k_{n}\geq 0$ and $\lambda=k_1e_1+k_2e_2+\cdots+k_{n}e_{n}$, $$\chi^{\ast}_{\lambda,W_{n}}=
e(k_1,k_2,\dots,k_{n})\in Y$$ and $$E(\chi^{\ast}_{\lambda,W_{n}})=x_{k_1}x_{k_2}\cdots x_{k_{n}}.$$ 
Given $f\in\bbQ[x_0,x_1,...]$, set \[\sigma(f)(x_0,x_1,...,x_{2n},x_{2n+1},...)=f(x_0,-x_1,...,x_{2n},
-x_{2n+1},...).\] Then, $\sigma$ is an involutive automorphism of $\bbQ[x_0,x_1,...]$. 

Write $a_{n}(\lambda)$, $b_{n}(\lambda)$, $c_{n}(\lambda)$, $d_{n}(\lambda)$ for the image of $j_{n}(F_{\Phi,
\lambda,W_{n}})$ under $E$ for $\Phi=\A_{n-1}$, $\B_{n}$, $\C_{n}$ or $\D_{n}$, and a weight $\lambda\in
\mathbb{Z}^{n}$. Observe that $a_{n}(\lambda)$, $b_{n}(\lambda)$, $c_{n}(\lambda)$, $d_{n}(\lambda)$ are homogeneous 
polynomials of degree $n$ with integer coefficients. Write $b'_{n}(\lambda)=\sigma(b_{n}(\lambda))$. 

\smallskip

For $n\geq 1$ and a weight $\lambda=a_1e_1+\cdots+a_{n}e_{n}$, define the matrices \begin{eqnarray*}&& 
\qquad\qquad\qquad\qquad\qquad A_{n}(\lambda)=(x_{|a_{j}+i-j|})_{n\times n},\\&& B_{n}(\lambda)\!=\!
(x_{|a_{j}+i-j|}\!-\!x_{|a_{j}+2n+1-i-j|})_{n\times n},\ B'_{n}(\lambda)\!=\!(x_{|a_{j}+i-j|}\!+\!x_{|a_{j}+
2n+1-i-j|})_{n\times n},\\&& C_{n}(\lambda)\!=\!(x_{|a_{j}+i-j|}\!-\!x_{|a_{j}+2n+2-i-j|})_{n\times n},\ 
D_{n}(\lambda)\!=\!(x_{|a_{j}+i-j|}\!+\!x_{|a_{j}+2n-i-j|})_{n\times n},\\&& \qquad\qquad\qquad
\qquad\qquad D'_{n}(\lambda)=(y_{i,j})_{n\times n},\end{eqnarray*} where $y_{i,j}=x_{|a_{j}+i-j|}\!+
\!x_{|a_{j}+2n-i-j|}$ if $i,j\leq n-1$, $y_{n,j}=\sqrt{2}x_{|a_{j}+n-j|}$, $y_{i,n}=
\frac{\sqrt{2}}{2}(x_{|a_{n}+i-n|}+x_{|a_{n}+n-i|})$ and $y_{n,n}=x_{|a_{n}|}$.

\begin{lemma}\label{L:polynomial-matrix}
We have \[\det A_{n}(\lambda)=a_{n}(\lambda),\quad \det B_{n}(\lambda)=b_{n}(\lambda),\quad
\det B'_{n}(\lambda)=b'_{n}(\lambda),\] \[\det C_{n}(\lambda)=c_{n}(\lambda),\quad
\frac{1}{2}\det D_{n}(\lambda)=\det D'_{n}(\lambda)=d_{n}(\lambda).\]
\end{lemma}

\begin{proof}
For $\Phi=A_{n-1}$, $\B_{n}$, $\C_{n}$, these follow from comparing $E(j_{n}([\lambda+\delta-w\delta]))$ 
($w\in W_{\Phi}$) with terms in the expansion of $\det A_{n}(\lambda)$, $\det B_{n}(\lambda)$, $\det C_{n}
(\lambda)$. Applying the involutive automorphism $\sigma$, we get $\det B'_{n}(\lambda)=b'_{n}(\lambda)$. 
For $\Phi=\D_{n}$, we define a new character $\epsilon': W_{n}\rightarrow\{1\}$ by $\epsilon'|_{W_{\D_{n}}}
=\epsilon|_{W_{\D_{n}}}$, and $\epsilon'(s_{e_1})=1$ (rather that $\epsilon(s_{e_1})=-1$). Due to 
$s_{e_{n}}(\delta_{\D_{n}})=\delta_{\D_{n}}$, we have $$F_{\D_{n},\lambda,W_{\D_{n}}}=\frac{1}{2}
\sum_{w\in W_{n}}\epsilon'(w)[\lambda+\delta-w\delta].$$ From this, we get $$\frac{1}{2}\det D_{n}(\lambda)
=d_{n}(\lambda).$$ Apparently, $\frac{1}{2}\det D_{n}(\lambda)=\det D'_{n}(\lambda)$.  
\end{proof}

\begin{proposition}\label{P:irreducible}
Assume $a_1\geq a_2\cdots\geq a_{n}\geq 0$, then each of $b_{n}(\lambda)$, $b'_{n}(\lambda)$, $c_{n}(\lambda)$, 
is an irreducible polynomial. Assume $a_1\geq a_2\cdots\geq|a_{n}|\geq 0$, then $d_{n}(\lambda)$ is an 
irreducible polynomial.
\end{proposition}

\begin{proof}
For $b_{n}(\lambda)$, the indeterminate appearing in it with largest index is $x_{a_1+2n-1}$, and $$b_{n}(\lambda)
=-x_{a_1+2n-1}b_{n-1}(\lambda')+\cdots,$$ where $\lambda'=(a_2,\dots,a_{n})$. By induction, $b_{n-1}(\lambda')$ 
is irreducible. Apparently, $b_{n-1}(\lambda')$ does not divide $b_{n}(\lambda)$, hence $b_{n}(\lambda)$ is 
irreducible. Similarly, $b'_{n}(\lambda)$, $c_{n}(\lambda)$, $d_{n}(\lambda)$ are irreducible. 
\end{proof}

\begin{proposition}\label{P:a=cd-a=bb}
If $n=2m+1$ is odd, $a_1\geq a_2\cdots\geq a_{n}$, and $a_{n+1-i}+a_{i}=0$, $\forall i$, $1\leq i\leq m$, 
then $$a_{2m+1}(\lambda)=c_{m}(\lambda_1)d_{m+1}(\lambda_2),$$ where $\lambda_{1}=(a_1,\dots,a_{m})$, 
$\lambda_{2}=(a_1,\dots,a_{m+1})$. 

If $n=2m$ is even,  $a_1\geq a_2\cdots\geq a_{n}\geq 0$, and $a_{n+1-i}+a_{i}=0$, $\forall i$, $1\leq i\leq m$, 
then $$a_{2m}(\lambda)=b_{m}(\lambda_1)b'_{m}(\lambda_2),$$ where $\lambda_{1}=\lambda_2=(a_1,\dots,a_{m})$.

For $n\geq 1$, assume $a_1\geq a_2\cdots\geq a_{n}$, then $a_{n}(\lambda)$ is reducible only if 
$a_{i}+a_{n+1-i}=0$ for any $i$, $1\leq i\leq n/2$.  
\end{proposition}

\begin{proof}
Define $L_{m}$ inductively by $L_{1}=1$ and \[L_{m}=\left(\begin{array}{ccc}&&1\\&L_{m-2}&\\1&&\end{array}
\right)\] for any $m\geq 2$. Then, $L_{m}^{2}=I$. The matrix $A_{2m}(\lambda)$ is of the form $$\left(
\begin{array}{cc}X&Y\\L_{m}YL_{m}&L_{m}XL_{m}\\\end{array}\right),$$ where $X,Y$ are two $m\times m$  
matrices. By calculation we have \begin{eqnarray*}&&\frac{1}{2}\left(\begin{array}{cc}I&L_{m}\\-L_{m}&I
\\\end{array}\right)\left(\begin{array}{cc}X&Y\\L_{m}YL_{m}&L_{m}XL_{m}\\\end{array}\right)
\left(\begin{array}{cc}I&-L_{m}\\L_{m}&I\\\end{array}\right)\\&&=\left(\begin{array}{cc}X+YL_{m}&0\\0
&L_{m}XL_{m}-L_{m}Y\\\end{array}\right).\end{eqnarray*} One can check that $X+YL_{m}$ ($X-YL_{m}$) is 
just the matrix $B'_{m}(\lambda_{2})$ ($B_{n}(\lambda_{1})$), thus $$a_{2m}(\lambda)=b_{m}(\lambda_1)
b'_{m}(\lambda_2).$$

The matrix $A_{2m+1}(\lambda)$ is of the form $$\left(\begin{array}{ccc}X&\beta^{t}&Y\\ \alpha&z&
\alpha L_{m}\\L_{m}YL_{m}&\gamma^{t}&L_{m}XL_{m}\\\end{array}\right),$$ where $X,Y$ are two $m\times m$ 
matrices, $\alpha,\beta,\gamma$ are $1\times m$ vectors. By calculation we have \begin{eqnarray*}&&
\frac{1}{2}\left(\begin{array}{ccc}I&&L_{m}\\&\sqrt{2}&\\-L_{m}&&I\\\end{array}\right)\left(\begin{array}
{ccc}X&\beta^{t}&Y\\\alpha&z&\alpha L_{m}\\L_{m}YL_{m}&\gamma^{t}&L_{m}XL_{m}\\\end{array}\right)\left(
\begin{array}{ccc}I&&-L_{m}\\&\sqrt{2}&\\L_{m}&&I\\\end{array}\right)\\&&=\left(\begin{array}{ccc}X+YL_{m}
&\frac{\sqrt{2}}{2}(\beta^{t}+L_{m}\gamma^{t})&0\\\sqrt{2}\alpha&z&0\\0&\frac{\sqrt{2}}{2}(-L_{m}\beta^{t}
+\gamma^{t})&L_{m}XL_{m}-L_{m}Y\\\end{array}\right).\end{eqnarray*} The matrix $$\left(\begin{array}{cc}X
+YL_{m}&\frac{\sqrt{2}}{2}(\beta^{t}+L_{m}\gamma^{t})\\\sqrt{2}\alpha&z\\\end{array}\right)$$ is just 
$D'_{m+1}(\lambda_{2})$, and the matrix $X-YL_{m}$ is just $C_{m}(\lambda_1)$. Thus, $$a_{2m+1}(\lambda)
=c_{m}(\lambda_1)d_{m}(\lambda_2).$$ 

We prove the last statement by induction on $n$. When $n=1$, it is trivial. When $n=2$, it is easy to 
check. Now we assume $n\geq 3$ and the statement is true for $1,2,..,n-1$, and show it in the $n$ case. 
First we show $a_1+a_{n}=0$. Suppose $a_1+a_{n}\neq 0$ and $a_{n}(\lambda)$ is reducible. By symmetry 
we may assume that $a_1+a_{n}\geq 1$. Among the indeterminates appearing in $a_{n}(\lambda)$, the one 
of largest index is $x_{a_1+n-1}$, and $$a_{n}(\lambda)=x_{a_1+n-1}a_{n-1}(\lambda')+\cdots,$$ where 
$\lambda'=(a_{2}-1,\dots,a_{n}-1)$. Write $a_{n}(\lambda)=(x_{a_{1}+n-1}u_1+v_1)u_2$ with $\deg u_2
\geq 1$. There are two cases according to $\deg u_1>0$ or $\deg u_1=0$. In the case of $\deg u_1>0$, 
we have $a_{n-1}(\lambda')=u_{1}u_{2}$. By induction, $a_{i}+a_{n+2-i}=2$, $\forall i$, $2\leq i\leq n$. 
In this case, $a_1+n-2\geq a_{2}+n-2>a_2+n-3=-(a_{n}+1-n)$. The sum of terms in $a_{n}(\lambda)$ dividing 
$x_{a_1+n-2}x_{a_{2}+n-2}$ is $x_{a_1+n-2}x_{a_{2}+n-2}a_{n-2}(\lambda'')$, where $\lambda''=(a_{2}-2,
\dots,a_{n}-2)$. The indeterminates $x_{a_1+n-2}$ and $x_{a_{2}+n-2}$ can only appear in $v_1$. Hence, 
$u_2$ is a factor of $a_{n-2}(\lambda'')$. If $n\geq 4$, then $a_{i}+a_{n+3-i}=4$, $\forall i$, 
$3\leq i\leq n$. This leads to $a_{3}>a_{2}$, hence a contradiction. If $n=3$, then 
$u_2=\pm{}a_{1}(\lambda'')$ and $a_{1}(\lambda'')|a_{2}(\lambda)$, which does not hold, hence a 
contradiction. In the case of $\deg u_1=0$, we may assume that $u_1=1$ and $u_2=a_{n-1}(\lambda')$. 
Due to $a_1+n-2\geq\max{|a_{j}+i-j|: 1\leq i\leq n-1, 2\leq j\leq n}$, then $v_1$ has a term 
$\pm{}x_{a_1+n-2}$. Then, the determinant of the remaining matrix of $\A_{n}(\lambda)$ by removing the 
last row and the first column (which is equla to $a_{n-1}(\lambda')$) is equal to the determinant of 
the remaining matrix of $\A_{n}(\lambda)$ by removing the $n-1$-th row and the first column. This is 
impossible as their leading terms are different. Now assume $a_1+a_{n}=0$, then $$a_{n}(\lambda)=
x_{a_1+n-1}x_{-(a_{n}+1-n)}a_{n-2}(\lambda''')+\cdots,$$ where $\lambda'''=(a_{2}-1,\dots,a_{n-1}-1)$. 
One can show that $a_{n-2}(\lambda''')$ does not divide $a_{n}(\lambda)$ (in case $n\geq 4$). Thus, 
$a_{n-2}(\lambda''')$ is reducible. By induction, $a_{i}+a_{n+1-i}=0$ for any $i$, $2\leq i\leq n/2$.  
\end{proof}

\medskip 

\noindent{\it Step 3, final conclusion.} 

By Proposition \ref{P:a=cd-a=bb}, $$a_{2n+1}(\lambda)=c_{n}(\lambda_1)d_{n+1}(\lambda_2).$$ This shows 
Equation \ref{Eq:A=CD}, hence the conclusion of Theorem \ref{T:tau-example}.

\subsection{Equalities and linear relations among $\tau$-dimension data}

Like in \cite[Section 4]{Yu-dimension}, Proposition \ref{P:dimension-character} leads us to study equalities 
among characters $$\{F_{\Phi,\lambda,\Aut(\Psi)}: \Phi\subset\Psi,\lambda\in\Lambda_{\Phi}(\mathbb{Q}\Psi)
\}$$ for a given root system $\Psi$, where the lattices $\mathbb{Q}\Psi$ and $\Lambda_{\Phi}(\mathbb{Q}\Psi)$ 
are as introduced in \cite[Pages 2687-2688]{Yu-dimension}). Similarly, if we want to study linear relations 
among $\tau$-dimension data for connected subgroups, then we need to study linear relations among the 
characters $\{F_{\Phi,\lambda,\Aut(\Psi)}: \Phi\subset\Psi,\lambda\in\Lambda_{\Phi}(\mathbb{Q}\Psi)\}$. 
As in \cite[Section 7]{Yu-dimension}, both studies may reduce to the case of $\Psi$ is an irreducible 
root system.  

When $\Psi=\BC_{n}$ and $\lambda\in\mathbb{Z}\Psi$, the equalities/linear relations among 
$F_{\Phi,\lambda,\Psi}$ correspond to multiplicative/algebraic relations among the polynomials 
$$\{a_{n}(\lambda),b_{n}(\lambda),c_{n}(\lambda),d_{n}(\lambda): n\geq 1\}.$$ Propositions \ref{P:irreducible} 
and \ref{P:a=cd-a=bb} give all such multiplicative relations. 

\begin{question}\label{Q:tau-dimension}
Is there a generalization of Theorem \cite[Theorem 1.5(2)]{An-Yu-Yu} to $\tau$-dimension data of more general 
weights? 

Moreover, can we find all generating relations of algebraic relations among the polynomials 
$\{a_{n}(\lambda),b_{n}(\lambda),c_{n}(\lambda),d_{n}(\lambda): n\geq 1\}$? 
\end{question}

When $\Psi=\B_{n}$, $\C_{n}$ or $\D_{n}$ and $\lambda\in\mathbb{Z}_{n}=\mathbb{Z}\BC_{n}$, it reduces to the 
$\Psi=\BC_{n}$ case, so is for $\Psi=\A_{n-1}$. However, for these root systems, the weight $\lambda$ is not 
necessarily $\lambda\in\mathbb{Z}_{n}$. To have a complete study we need to consider polynomials associated 
to non-integral weights. We don't know if this arises new complication, or gives new interesting phenomenon.  

\smallskip 

When $\Psi$ is an exceptional irreducible root system, from \[A_{\Phi,\lambda}=\sum_{w\in W_{\Phi}}\sgn(w)
[\lambda+\delta_{\Phi}-w\delta_{\Phi}]\] and \[F_{\Phi,\lambda,\Aut(\Psi)}=\frac{1}{|\Aut(\Psi)|}\sum_{\gamma
\in\Aut(\Psi)}\gamma(A_{\Phi,\lambda}),\] the orbits $\Aut(\Psi)\lambda$ and $\Aut(\Psi)(\lambda+2\delta_{\Phi})$ 
are determined, as well as the polynomial $$\sum_{w\in W_{\Phi}}\sgn(w)t^{|\lambda+\delta_{\Phi}-
w\delta_{\Phi}|^{2}}=t^{|\lambda|^{2}}\prod_{\alpha\in\Phi^{+}}(1-t^{(\lambda+\delta_{\Phi},2\alpha)})$$ 
(cf. \cite[Proposition 5.2]{Yu-dimension}). These invariants all involve the weight $\lambda$. The only 
invariant without involving $\lambda$ which could be extracted from them easily is $\prod_{\alpha\in\Phi^{+}}
(1-t^{(\lambda+\delta_{\Phi},2\alpha)})$, but this seems too weak. We do not know how to use the three 
invariants in an effective way, so as to find all equalities among $\{F_{\Phi,\lambda,\Aut(\Psi)}: 
\Phi\subset\Psi,\lambda\in\Lambda_{\Phi}(\mathbb{Q}\Psi)\}.$ We have no idea yet for how to study linear 
relations among the characters $\{F_{\Phi,\lambda,\Aut(\Psi)}: \Phi\subset\Psi,\lambda\in\Lambda_{\Phi}
(\mathbb{Q}\Psi)\}.$

\subsection{Isospectral hermitian vector bundles}  

Let $H$ be a closed subgroup of a connected compact Lie group $G$, and $(V_{\tau},\tau)$ ($V_{\tau}$ 
is the representation space of $\tau\in\widehat{H}$) be a finite-dimensional irreducible complex linear 
representation of $H$. Write $E_{\tau}=G\times_{\tau}V_{\tau}$ for a $G$-equivariant vector bundle on 
$X=G/H$ induced from $V_{\tau}$. As a set, $E_{\tau}$ is the set of equivalence classes in 
$G\times V_{\tau}$, $$(g,v)\sim(g',v')\Leftrightarrow\exists x\in H\textrm{ s.t. } g'=gx,\ v'=x^{-1}
\cdot v.$$ Write $C^{\infty}(G/H,E_{\tau})$ for the space of smooth sections of $E_{\tau}$. Then, 
$$C^{\infty}(G/H,E_{\tau})=(C^{\infty}(G,V_{\tau}))^{H},$$ where $C^{\infty}(G,V_{\tau})$ the space 
of smooth functions $f: G\rightarrow V_{\tau}$ and $H$ acts on it through $$(xf)(g)=x\cdot f(gx).$$ 
The group $G$ acts on $C^{\infty}(G/H,E_{\tau})$ through $$(g'f)(g)=f(g'^{-1}g).$$ By differentiation, 
we get an action of $\mathfrak{g}_{0}=\Lie G$ on $C^{\infty}(G/H,E_{\tau})$, and so an action of the 
universal enveloping algebra $U(\mathfrak{g}_{0})$ on $C^{\infty}(G/H,E_{\tau})$. Let $\Delta_{\tau}$ 
denote the resulting differential operator on $C^{\infty}(G/H,E_{\tau})$ from the Casimir element 
$C\in Z(\mathfrak{g}_{0})=U(\mathfrak{g}_{0})^{G}$. The action of $\Delta_{\tau}$ on $C^{\infty}(G/H,
E_{\tau})$ commutes with the action by $G$, and it is an second order elliptic differential operaotr. 

Choose an $H$-invariant unitary form $\langle\cdot,\cdot\rangle$ on $V_{\tau}$ (which is unique up to 
a scalar). It induces a hermitian metric on $E_{\tau}$ and makes it a hermitian vector bundle\footnotemark. 
Define a hermitian pairing $(\cdot,\cdot)$ on $C^{\infty}(G/H,E_{\tau})$ by $$(f_1,f_2)=\int_{G/H}\langle 
f_{1}(g),f_{2}(g)\rangle d(gH),$$ where $d(gH)$ is a $G$-equivariant measure on $G/H$ of volume $1$. 
As $\Delta_{\tau}$ is an elliptic differenrial operator, any eightform of it in $L^{2}(G/H,E_{\tau})$ 
is a smooth section. By the Peter-Weyl theorem, $$L^{2}(G/H,E_{\tau})=\hat{\bigoplus}_{\rho\in\widehat{G}}
L^{2}(G/H,E_{\tau})_{\rho}$$ and the $\rho$-isotropic subspace, $L^{2}(G/H,E_{\tau})_{\rho}=
\Hom(\tau,\rho|_{H})\otimes_{\mathbb{C}}\rho.$ Also, we know that $\Delta_{\tau}$ acts on the 
$\rho$-isotropic component $L^{2}(G/H,E_{\tau})_{\rho}$ by a scalar determined by $\rho$. By this, we 
have the following fact: if $\mathscr{D}_{H_1,\tau_1}=\mathscr{D}_{H_2,\tau_2}$, then the Hermitian 
vector bundles $E_{\tau_1}=G\times_{\tau_1}V_{\tau_{1}}$ (on $G/H_1$) and $E_{\tau_2}=G\times_{\tau_2}
V_{\tau_2}$ (on $G/H_2$) are isospectral. Here, isospectral means the multiplicities of eigenspaces 
of $\Delta_{\tau_1}$ and $\Delta_{\tau_2}$ are the same for any eigenvalue. 
\footnotetext{We could define a Laplace-Beltrami operator on $C^{\infty}(G/H,E_{\tau})$ from the hermitian 
metric on $E_{\tau}$ such that it coincides with $\Delta_{\tau}$ up to a scalar.}

\begin{corollary}\label{C:isomorphic}
For $G=\SU(4n+2)$, subgroups $H_1,H_2$ and representations $\tau_{\lambda}$ and $\tau_{\lambda'}$ as in 
Theorem \ref{T:tau-example}, the hermitian vector bundles $E_{\tau_{\lambda}}=G\times_{\tau_{\lambda}}
V_{\tau_{\lambda}}$ (on $G/H_1$) and $E_{\tau_{\lambda'}}=G\times_{\tau_{\lambda'}}V_{\tau_{\lambda'}}$ 
(on $G/H_2$) are isospectral. 
\end{corollary}

\subsection{Generalization of a theorem of Larsen-Pink}\label{SS:Taylor}

Langlands associated a (conjectural) group $H^{\pi}$ to an automorphic form $\pi$ (\cite{Langlands}). If the 
Langlands group $L_{F}$ exists (which is a much bigger conjecture), then $H^{\pi}$ should be the image of the 
Langlands $L$-homomorphism for $\pi$. Besides dimension datum, it is interesting to see if more invariant 
theory data could determine $H^{\pi}$. The following proposition\footnotemark concerns this. In case $H$ is 
semisimple, it is \cite[Theorem 1]{Larsen-Pink}. 
\footnotetext{Proposition \ref{P:Taylor} is proposed by Professor Richard Taylor. The author would like to 
thank him for discussion about dimension datum and related subjects.} 

\begin{proposition}\label{P:Taylor}
let $G,H$ be connected compact Lie groups, and $f_1,f_2: H\rightarrow G$ be two homomorphisms. If 
$$\dim((\rho\circ f_1)\otimes\chi)^{H}=\dim((\rho\circ f_2)\otimes\chi)^{H}$$ for any $\rho\in\widehat{G}$ 
and any $\chi\in\mathcal{X}(H)=\Hom(H,\U(1))$, then $f_1(H)\cong f_2(H)$.  
\end{proposition}

\begin{proof}[Proof of Proposition \ref{P:Taylor}] 

\noindent{\it The tours case.} To motivate the proof in the general case, we first show Proposition 
\ref{P:Taylor} {\it in the case of $H$ is a torus.} First we show $\ker f_{1}=\ker f_2$. Suppose no. 
Without loss of generality we assume that $\ker f_{1}\not\subset\ker f_2$. Then, there exists $\chi\in
\mathcal{X}(H)$ such that $\chi|_{\ker f_{1}}\neq 1$ and $\chi|_{\ker f_{2}}=1$. For any $\rho\in
\widehat{G}$, $\rho\circ f_{1}|_{\ker f_{1}}=1$, hence $\dim((\rho\circ f_1)\otimes\chi)^{H}=0$. 
As $\chi|_{\ker f_{2}}=1$, $\chi$ descends to a linear character $\chi'$ of $f_{2}(H)\subset G$. Choose 
some $\rho\in\widehat{G}$ such that $\rho\subset\Ind_{f_{2}(H)}^{G}(\chi'^{\ast}).$ Then, $\dim((\rho
\circ f_2)\otimes\chi)^{H}>0$. This is in contradiction with $\dim((\rho\circ f_1)\otimes\chi)^{H}=
\dim((\rho\circ f_2)\otimes\chi)^{H}$ for any $\rho\in\widehat{G}$. Thus, we have showed $\ker f_{1}=
\ker f_2$. 

By considering $H/\ker f_{1}$ instead, we may assume that both $f_{1}$ and $f_{2}$ are injections. By 
considering the support of the Sato-Tate measures of $f_{1}(H)$ and $f_{2}(H)$, we know that $f_{1}(H)$ 
and $f_{2}(H)$ are conjugate in $G$ (\cite[Proposition 3.7]{Yu-dimension}). We may assume that 
$f_{1}(H)=f_{2}(H)$, and denote it by $T$. Write $\Gamma^{\circ}=N_{G}(T)/Z_{G}(T)$. 

We identify $H$ with $T$ through $f_1$, and regard $f_{2}$ as an automorphism of $T$, denoted by $\phi$. 
Then, the condition is equivalent to: $$F_{\emptyset,\chi,\Gamma^{0}}=F_{\emptyset,\phi^{\ast}(\chi),
\Gamma^{0}}.$$ Which is equivalent to: $\phi^{\ast}(\chi)\in\Gamma^{\circ}\cdot\chi$. We show that 
$\phi=\gamma|_{T}$ for some $\gamma\in\Gamma^{\circ}$. Suppose it is not the case. For any $\gamma\in
\Gamma^{0}$, due to $\phi\neq\gamma|_{\Gamma^{0}}$, $$X_{\gamma}=\{\chi\in\mathcal{X}(H):\phi^{\ast}
(\chi)=\gamma\cdot\chi\}$$ is a sublattice of $\mathcal{X}(H)$ with positive corank. Hence, 
$$\bigcup_{\gamma\in\Gamma^{\circ}}X_{\gamma}\neq\mathcal{X}(H).$$ This is in contradiction with 
$\phi^{\ast}(\chi)\in\Gamma^{\circ}\cdot\chi$ for any $\chi\in\mathcal{X}(H)$. 

\smallskip 

\noindent{\it The general case.} First we show $H_{der}\ker f_1=H_{der}\ker f_2$. Here $H_{der}=[H,H]$ 
is the derived subgroup of $H$. Suppose no. Without loss of generality we assume that $H_{der}\ker f_1
\not\subset H_{der}\ker f_2$. Then, there exists $\chi\in\mathcal{X}(H)$ such that 
$\chi|_{H_{der}\ker f_{1}}\neq 1$ and $\chi|_{H_{der}\ker f_{2}}=1$. For any $\rho\in\widehat{G}$, 
$\rho\circ f_{1}|_{\ker f_{1}}=1$, hence $\dim((\rho\circ f_1)\otimes\chi)^{H}=0$. 
As $\chi|_{H_{der}\ker f_{2}}=1$, $\chi$ descends to a linear character $\chi'$ of $f_{2}(H)\subset G$. 
Choose some $\rho\in\widehat{G}$ such that $\rho\subset\Ind_{f_{2}(H)}^{G}(\chi').$ Then, 
$\dim((\rho\circ f_2)\otimes\chi)^{H}>0$. This is in contradiction with $\dim((\rho\circ f_1)\otimes
\chi)^{H}=\dim((\rho\circ f_2)\otimes\chi)^{H}$ for any $\rho\in\widehat{G}$. Thus, we have showed 
$$H_{der}\ker f_1=H_{der}\ker f_2.$$ 

Write $H_{i}=f_{i}(H)$. Due to $H/H_{der}\ker f_{i}\cong H_{i}/(H_{i})_{der}$, we have 
$$H_{1}/(H_{1})_{der}\cong H_{2}/(H_{2})_{der}.$$ Choose a maximal torus $T_{i}$ of $H_{i}$ ($i=1$ or 
$2$). Write $(T_{i})_{s}=T_{i}\cap(H_{i})_{der}$. Then, $$T_{i}=Z(H_{i})^{0}\cdot(T_{i})_{s}.$$ 
Due to $T_{i}/(T_{i})_{s}\cong H_{i}/(H_{i})_{der}$, we have $$T_{1}/(T_{1})_{s}\cong T_{2}/(T_{2})_{s}.$$ 
By considering the support of the Sato-Tate measures of $H_1$ and $H_{2}$, we know that $T_{1}$ and 
$T_{2}$ are conjugate in $G$ (\cite[Proposition 3.7]{Yu-dimension}). We may assume that $T_{1}=T_{2}$, 
and denote it by $T$. Write $\Gamma^{\circ}=N_{G}(T)/Z_{G}(T)$. 

Choose a biinvariant Riemannian metric on $G$, which induces a $\Gamma^{\circ}$ invariant inner product 
on the Lie algebra of $T$, and also a $\Gamma^{\circ}$ invariant inner product on the weight lattice 
$\Lambda_{T}$. Write $\Phi_1\subset\Lambda_{T}$ ($\Phi_2\subset\Lambda_{T}$) for root system of $H_1$ 
(of $H_2$). Write $$X_{i}=\mathcal{X}(T_{i}/(T_{i})_{s})\subset\Lambda_{T}.$$ Due to $T_{1}/(T_{1})_{s}
\cong T_{2}/(T_{2})_{s}$, we have an isomorphism $\phi: X_{1}\rightarrow X_{2}$. For any $\chi_1\in X_1$, 
write $\chi_2=\phi(\chi_1)$. Due to $W_{\Phi_{i}}\subset\Gamma^{0}$ ( $i=1,2$) and the $\Gamma^{0}$ 
invariance the inner product on the Lie algebra of $T$, the Lie algebra of $Z(H_{i})^{0}$ is orthogonal 
to the Lie algebra of $(T_{i})_{s}$. By the condition in the question and Proposition 
\ref{P:dimension-character}, $$F_{\Phi_1,\chi_1,\Gamma^{0}}=F_{\Phi_2,\chi_2,\Gamma^{0}}.$$ Due to 
$\chi_{i}$ is orthogonal to $\delta_{\Phi_{i}}-w\delta_{\Phi_{i}}$ for any $w\in W_{\Phi}$, 
$\chi^{\ast}_{\chi_{i},\Gamma^{\circ}}$ is the shortest term in the expansion of 
$F_{\Phi_{i},\chi_{i},\Gamma^{0}}$. Thus, $\chi_2=\gamma\cdot\chi_{1}$ for some $\gamma\in\Gamma^{\circ}.$ 
Arguing similarly as in the torus case, one can show that $\phi=\gamma|_{X_{1}}$ for some $\gamma\in
\Gamma^{\circ}$. By this, we may assume that $\phi=\id$. That is to say, $X_{1}=X_{2}$ and 
$(T_{1})_{s}=(T_{2})_{s}$. As the Lie algebra of $Z(H_{i})^{0}$ is orthogonal to the Lie algebra of 
$(T_{i})_{s}$, we also have $Z(H_{1})^{0}=Z(H_{2})^{0}$. Write $Z=Z(H_{i})^{0}$, $T_{s}=(T_{i})_{s}$, 
and $X=X_{i}$. Write $G'=Z_{G}(Z)$ and $$\Gamma'=N_{G'}(T_{s})(T_{s})/Z_{G'}(T_{s}).$$ 

There is an identification\footnotemark $$\Gamma'=\{\gamma\in\Gamma^{\circ}: \gamma|_{Z}=\id\}=\{\gamma
\in\Gamma^{\circ}: \gamma|_{X}=\id\}.$$ For any $\gamma\in\Gamma^{\circ}-\Gamma'$, \footnotetext{Here 
we use the fact of the Lie algebra of $Z$ is orthogonal to the Lie algebra of $T_{s}$.} $$X_{\gamma}=
\{\chi\in X:\gamma\cdot\chi=\chi\}$$ is a sublattice of positive corank. Thus, $$\bigcup_{\gamma\in
\Gamma^{\circ}}X_{\gamma}\neq X.$$ Choose $\chi_{0}\in X-\bigcup_{\gamma\in\Gamma^{\circ}}X_{\gamma}$. 
Write $$c=\min\{|\gamma\cdot\chi_{0}-\chi_{0}|:\gamma\in\Gamma^{\circ}-\Gamma'\}>0$$ and $$c'=\max
\{|\delta_{\Phi_{2}}-w_{2}\delta_{\Phi_2}-\gamma(\delta_{\Phi_{1}}-w_{1}\delta_{\Phi_1})|: w_{1}\in 
W_{\Phi_1}, w_{2}\in W_{\Phi_2},\gamma\in\Gamma^{\circ}\}\geq 0.$$ Take $m\geq 1$ such that $mc>c'$. 
Put $\chi=m\chi_0$. Then, for $\gamma\in\Gamma^{\circ}$, $w_{1}\in W_{\Phi_1}$, $w_{2}\in W_{\Phi_2}$, 
$$\gamma(\chi+\delta_{\Phi_{1}}-w_{1}\delta_{\Phi_1})=\chi+\delta_{\Phi_{2}}-w_{2}\delta_{\Phi_2}$$ 
if and only if $\gamma\in\Gamma'$ and $$\gamma(\delta_{\Phi_{1}}-w_{1}\delta_{\Phi_1})=
\delta_{\Phi_{2}}-w_{2}\delta_{\Phi_2}.$$ Then, the equation $$F_{\Phi_1,\chi,\Gamma^{0}}=
F_{\Phi_2,\chi,\Gamma^{0}}$$ implies $$F_{\Phi_1,0,\Gamma'}=F_{\Phi_2,0,\Gamma'}.$$ 
By \cite[Proposition 3.8]{Yu-dimension}, this means the subgroups $(H_{1})_{der}$ and $(H_{2})_{der}$ 
of $G'$ have the same dimension datum. By \cite[Theorem 1]{Larsen-Pink}, $(H_1)_{der}$ is isomorphic 
to $(H_2)_{der}$. A more detailed argument using the method in \cite{Yu-dimension} shows that this 
isomorphism could extend to an isomorphism from to $H_1$ and $H_2$\footnotemark, hence finish the 
proof. \footnotetext{Shortly to say, the argument goes in this way: define $\Psi_{T_{s}}$ as in 
\cite[Definition 3.1]{Yu-dimension}. Then, $\Gamma'\subset\Aut(\Psi_{T_{s}})$. Thus, 
$$F_{\Phi_1,0,\Aut(\Psi_{T_{s}})}=F_{\Phi_2,0,\Aut(\Psi_{T_{s}})}.$$ By this, results in 
\cite[Section 7]{Yu-dimension} imply that $\Phi_2=\gamma\cdot\Phi_1$ for some $\gamma\in
\Aut(\Psi_{T_{s}})$. This leads to an isomorphism $\eta: (H_{1})_{der}\rightarrow (H_{2})_{der}$ 
which stabilzies $T_{s}$ with $\eta|_{T_s}=\gamma$. We have $$Z\cap(H_1)_{der}\subset T_{s}\cap 
Z(G').$$ Decompose $\Psi_{T_{s}}$ into an orthogonal union of irreducible root systems, which 
corresponds to a decomposition of $T_{s}$. Due to $\BC_{n}$ is its own dual lattice, $T_{s}\cap 
Z(G')$ is contained in the product of those factors of $T_{s}$ corresponding to reduced 
irreducible factors of $\Psi_{T_{s}}$. The results in \cite[Section 7]{Yu-dimension} imply that 
there exists $\gamma'\in\Gamma'$ such that $\gamma'^{-1}\gamma$ acts trivially on reduced 
irreducible factors of $\Psi_{T_{s}}$. Hence, $$\eta|_{T_{s}\cap Z(G')}=\gamma|_{T_{s}\cap Z(G')}
=\gamma'|_{T_{s}\cap Z(G')}=\id.$$ Defining $\eta|_{Z}=\id$, then $\eta$ extends to an 
isomorphism $\eta: H_1\rightarrow H_2$.}    
\end{proof}

\section{Dimension datum of a disconnected subgroup}\label{S:ARS}

After the papers \cite{Larsen-Pink}, \cite{An-Yu-Yu}, \cite{Yu-dimension}, we have a pretty well understanding 
of dimension data of connected closed subgroups of a compact Lie group. In this section we study dimension 
data of disconnected closed subgroups. Extending the strategy in \cite{Yu-dimension}, we transfer the study 
of dimension data to the the study of certain characters, supported on {\it maximal commutative connected 
subsets}. We introduce a kind of {\it affine root system}. The character is associated to affine root system, 
and we use data from affine root system to get clear expression for the character. The final formula is the 
same as in the torus case. 

\subsection{Generalized Cartan subgroup and Weyl integration formula}\label{SS:integration}

Let $G$ be a compact Lie group. A closed abelian subgroup $S$ of $G$ is called a {\it (generalized) Cartan
subgroup} if $S$ contains a dense cyclic subgroup and $$S^{0}=(Z_{G}(S))^{0}$$ (\cite[Definition 4.1]{BtD}). 
We call a closed commutative connected subset $S'$ of $G$ a {\it maximal commutative connected subset} 
if $$s^{-1}S'=(Z_{G}(S'))^{0}$$ for any $s\in S'$. There is a close relationship between generalized 
Cartan subgroups and maximal commutative connected subsets: if $S'$ is a maximal commutative connected 
subset of $G$, then $S=\langle S'\rangle$ is a generalized Cartan subgroup; if $S$ is a generalized Cartan 
subgroup of $G$ and $s\in S$ generates $S/S^{0}$, then $S'=sS^{0}$ is a maximal commutative connected 
subset. For any element $g\in G$, choose a maximal torus $T^{g}$ of $Z_{G}(g)^{0}$. Set $$S=\langle T^{g},
g\rangle$$ and $$S'=gT^{g}.$$ Then, $S$ is a generalized Cartan subgroup of $G$, and $S'$ is a maximal 
commutative connected subset of $G$. Moreover, all generalized Cartan subgroups (and maximal commutative 
connected subsets) of $G$ arise in this way. 

For a generalized Cartan subgroup $S$ in $G$, $$S^{0}\subset S\cap G^{0}\subset Z_{G^{0}}(S).$$ We would  
like to remind that in general $S^{0}\neq S\cap G^{0}$ (i.e, $S\cap G^{0}$ is not necessarily connected)  
and $S\cap G^{0}\neq Z_{G^{0}}(S)$. 

\begin{example}\label{E:A3}
Set $G=(\SU(4)/\langle-I\rangle)\rtimes\langle\sigma\rangle$, where $\sigma^2=[iI]$, and $$\sigma[X]
\sigma^{-1}=[L\overline{X}L^{-1}]$$ for $L=\left(\begin{array}{cccc}0&0&1&0\\0&0&0&1\\1&0&0&0\\0&1&0&0
\end{array}\\\right).$ Write $$T=\{\diag\{\lambda_{1},\lambda_{2},\lambda_{1}^{-1},\lambda_{2}^{-1}\}:
|\lambda_{1}|=|\lambda_{2}|=1\},$$ and $S=\langle T,\sigma\rangle.$ Then, $S$ is a generalized Cartan 
subgroup of $G$. In this case, $$\sigma^2=[iI]\in S\cap G^{0}-S^{0}.$$ Thus, $S^{0}\neq S\cap G^{0}$.  
\end{example}

\begin{example}\label{E:A3-2}
Set $G=(\SU(4)/\langle-I\rangle)\rtimes\langle\sigma\rangle$, where $\sigma^2=1$, and $$\sigma[X]
\sigma^{-1}=[L\overline{X}L^{-1}]$$ for $L=\left(\begin{array}{cccc}0&0&1&0\\0&0&0&1\\1&0&0&0\\0&1&0&0
\end{array}\right).$ Write $$T=\{\diag\{\lambda_{1},\lambda_{2},\lambda_{1}^{-1},\lambda_{2}^{-1}\}:
|\lambda_{1}|=|\lambda_{2}|=1\},$$ and $S=\langle T,\sigma\rangle.$ Then, $S$ is a generalized Cartan 
subgroup of $G$. In this case, $$[iI]\in Z_{G^{0}}(S)-S\cap G^{0}.$$ Thus, $S\cap G^{0}\neq Z_{G^{0}}(S)$.  
\end{example}

Given a maximal commutative connected subset $S'$ of $G$, set $$W(G,S')=N_{G^{0}}(S')/s^{-1}S'$$ ($s\in 
S'$) and call it the {\it Weyl group} of $S'$. Given a generalized Cartan subgroup $S$ of $G$, set 
$$W(G,S)=N_{G^{0}}(S)/S^{0}$$ and call it the {\it Weyl group} of $S$. Both $W(G,S')$ and $W(G,S)$ 
are finite groups. If $S'$ is a maximal commutative connected subset and $S=\langle S'\rangle$, then 
$S$ is a generalized Cartan subgroup, and $N_{G^{0}}(S')=N_{G^{0}}(S)$. Thus, $W(G,S')=W(G,S)$. 
The group $W(G,S')$ ($=W(G,S')$) acts on $S'$ (and on $S$) through conjugation.   

\begin{proposition}(\cite[Propositions 4.2, 4.3, 4.7]{BtD})\label{P:compact-conjugacy}
Given a compact Lie group $G$ and a connected component $gG^{0}$ of $G$, any two maximal commutative 
connected subsets in $gG^{0}$ are conjugate. 

If $S'$ is a maximal commutative connected subset in $gG^{0}$, then every $G^{0}$ conjugacy class in 
$gG^{0}$ contains an element in $S'$. Two elements in $S'$ are in a $G^{0}$ conjugacy class if and 
only if they are in the same $W(G,S')$ orbit.  
\end{proposition} 

By Proposition \ref{P:compact-conjugacy}, $G^{0}$ conjugacy classes in a connected component $gG^{0}$ are 
parameterized by $S'/W(G,S')$.

\smallskip

Let $S'$ be a maximal commutative connected subset of $G$. Write $S=\langle S'\rangle$. Choose $s_{0}
\in S'$. 

\begin{lemma}\label{L:S'-q}
The map $$q: G^{0}/S^{0}\times S'\rightarrow s_{0}G^{0},\quad (gS^{0},s)\mapsto gsg^{-1}$$ is surjective, 
and the degree of the map $q$ is equal to $|W(G,S')|$ 
\end{lemma}
\begin{proof}
Surjectivity of the map $q$ is shown in \cite[Lemma 4.5]{BtD}. It is shown in \cite[Lemma 2.2]{Wendt}) that 
the degree of the map $q$ is equal to $|W(G,S')|$. 
\end{proof}

Write $dg$ for a left $G^{0}$ invariant measure of volume one on $s_{0}G^{0}$, $d\bar{g}$ for a $G^{0}$ 
invariant measure on $G^{0}/S^{0}$ of volume one, and $ds$ for a left $S^{0}$ invariant measure on $S'$ 
of volume one. Then, $q^{\ast}dg=\det(q)d\bar{g}\wedge ds$ for some positive-valued function $\det(q)$. 
Write $\frg$ for the complexified Lie algebra of $G$ and $\frs$ for the complexified Lie algebra of 
$S^{0}$. Then, we have the following formula for $\det(q)$.

\begin{lemma}\label{L:det(q)}
For any $g\in G^{0}$ and $s\in S'$, $$\det(q)(gS^{0},s)=\det(\Ad(s)-1)|_{\frg/\frs}.$$
\end{lemma}

\begin{proof}
This is  \cite[Lemma 2.1]{Wendt}. 
\end{proof}

Write \begin{equation}\label{Eq:density}D(s)=\frac{1}{|W(G,S')|}\det(\Ad(s)-1)|_{\frg/\frs},\ s\in S',
\end{equation} and call it the {\it density function} on $S'$. 

\begin{proposition}\label{P:integration}
For any $G^{0}$ conjugation invariant continuous function $f$ on $s_{0}G^{0}$, $$\int_{s_{0}G^{0}}f(g)dg
=\int_{S'}f(s)D(s)ds.$$
\end{proposition}

\begin{proof}
This is  \cite[Proposition 2.3]{Wendt}. 
\end{proof}

Choose a maximal commutative connected subset in each connected component of $G$. Write $S'_1$, $S'_2$,
..., $S'_{m}$ for these chosen subsets. For each $i$, write $D_{i}(s)$ for the corresponding density 
function as in (\ref{Eq:density}). Write $dg$ for a normalized Haar measure on $G$. For each $i$, we 
denote by $ds$ a left $s_{i}^{-1}S'_{i}$ ($s_{i}\in S'_{i}$) invariant measure of volume one on $S'_{i}$. 
The following is Weyl integration formula on $G$. It follows directly from Proposition \ref{P:integration} 
directly.   

\begin{proposition}\label{P:integration2}
For any $G^{0}$ conjugation invariant continuous function $f$ on $G$, $$\int_{G}f(g)dg=
\frac{1}{|G/G^{0}|}\sum_{1\leq i\leq m}\int_{S'_{i}}f(s)D(s)ds.$$
\end{proposition}

\subsection{Affine root datum}

Now we endow $G$ with a biinvariant Riemannian metric, which induces a positive definite inner product 
on the character group of $S^{0}$, denoted by $(\cdot,\cdot)$. For two characters $\lambda$ and $\mu$ 
of $S$, define $$(\lambda,\mu)=(\lambda|_{S^{0}},\mu|_{S^{0}}).$$ We introduce an affine root system 
$(R(G,S),R^{\vee})$ from the conjugation action of $S$ on $\mathfrak{g}$. 


The conjugation action of $S$ on $\frg$ gives a decomposition $$\frg=\sum_{\lambda\in X^{\ast}(S)}
\frg_{\lambda},$$ where $$\frg_{\lambda}=\{Y\in\frg: \Ad(s)Y=\lambda(s)Y,\ \forall s\in S\}.$$ Since 
$(Z_{G}(S))^{0}=S^{0}$, the zero-weight space is equal to the complexified Lie algebra of $S$. Set 
$$R=R(G,S)=\{\alpha\in X^{\ast}(S)-\{0\}:\frg_{\alpha}\neq 0\}.$$ An element $\alpha\in R(G,S)$ is 
called a {\it root} and $\frg_{\alpha}$ is called the {\it root space} for a root $\alpha$. A root 
$\alpha$ is called an {\it infinite root} if $\alpha|_{S^{0}}\neq 0$; it is called a {\it finite 
root} if $\alpha|_{S^{0}}=0$.


\begin{lemma}\label{L:infinite1}
Let $\alpha$ be an infinite root. Then $\dim\frg_{\alpha}=1$ and $2\alpha$ is not a root. 
\end{lemma}

\begin{proof} 
Set $$G^{[\alpha]}=Z_{G}(\ker\alpha)$$ and $$G_{[\alpha]}=[(G^{[\alpha]})^{0},(G^{[\alpha]})^{0}].$$ 
Write $T_{[\alpha]}=G_{[\alpha]}\cap S$. Then, $T_{[\alpha]}$ is a maximal torus of $G_{[\alpha]}$, 
and $\dim T_{[\alpha]}=1$. Thus, the complexified Lie algebra of $G_{[\alpha]}$ is semisimple and 
is of rank one. Hence, it is isomorphic to $\mathfrak{sl}_{2}(\mathbb{C})$. By this, $\dim\frg_{\alpha}
=1$ and $2\alpha$ is not a root. 
\end{proof}

There is a unique cocharacter $\check{\alpha}\in X_{\ast}(S)$ whose image lies in $T_{[\alpha]}$ and 
such that $\langle\alpha,\check{\alpha}\rangle=2$. Define \begin{equation}\label{Eq:reflection}
s_{\alpha}(x)=x\check{\alpha}(\alpha(x)^{-1}),\ \forall x\in S.\end{equation} The following proposition 
follows from the classical $\mathfrak{sl}_{2}$ theory (cf. \cite{Knapp} and \cite{BtD}).  

\begin{lemma}\label{L:reflection}
There exists $$n_{\alpha}\in N_{G_{[\alpha]}}(T_{[\alpha]})=G_{[\alpha]}\cap N_{G}(A)$$ such that
$\Ad(n_{\alpha})|_{A}=s_{\alpha}.$

We have $s_{\alpha}|_{\ker\alpha}=\id$, $s_{\alpha}|_{\Im\check{\alpha}}=-1$ and $s_{\alpha}^2=1$.
\end{lemma}

\smallskip 

Given a finite root $\alpha$ of order $n$, set $$G^{[\alpha]}=Z_{G}(\ker(\alpha)).$$ Its complexified 
Lie algebra is $$\frg^{[\alpha]}=\mathfrak{s}\oplus(\bigoplus_{m\in\mathbb{Z}^{\ast}}\frg_{m\alpha}).$$

\begin{lemma}\label{L:abelian}
The subalgebra $\frg^{[\alpha]}$ is abelian.
\end{lemma}
\begin{proof}
Write $\frh=Z_{\frg}(S^0)$, and $\frh=z(\frh)\oplus[\frh,\frh]$ for the Levi decomposition of $\frh$. 
Then, $\mathfrak{s}\subset z(\frh)$. Since $\alpha$ is a finite root, $S^0\subset\ker\alpha$. By this, 
$$\frg^{[\alpha]}\subset\frh=z(\frh)\oplus[\frh,\frh].$$ As the conjugation action of $S$ on $\frh$ 
stabilizes both $z(\frh)$ and $[\frh,\frh]$, we have $$\frg^{[\alpha]}=(\frg^{[\alpha]}\cap z(\frh))
\oplus(\frg^{[\alpha]}\cap[\frh,\frh]).$$ Choose an element $y\in S$ generates $S/\ker\alpha$. Then, 
$$\frs=Z_{\frg}(S)=(\frg^{[\alpha]})^{y}.$$ Apparently, $\frs\subset\frg^{[\alpha]}\cap z(\frh)$. Hence, 
$(\frg^{[\alpha]}\cap[\frh,\frh])^{y}=0$. By a theorem of Borel, this indicates that $\frg^{[\alpha]}\cap
[\frh,\frh]$ is abelian. Therefore, $\frg^{[\alpha]}\subset z(\frh)\oplus(\frg^{[\alpha]}\cap[\frh,
\frh])$ is also abelian. 
\end{proof}

Let $$\frg_{[\alpha]}=\bigoplus_{m\in\mathbb{Z}^{\ast}}\frg_{m\alpha},$$ and $G_{[\alpha]}$ be the image of 
$\frg_{[\alpha]}$ under the exponential map. 

\begin{lemma}\label{L:closed}
$G_{[\alpha]}$ is a closed abelian subgroup of $G$ and $G_{[\alpha]}\cap S$ is a finite group.
\end{lemma}

\begin{proof}
Choose $y\in S$ generates $S/\ker\alpha$. The conjugation action of $y$ on the torus $(G^{[\alpha]})^0$ 
is given by some $y_{\ast}\in\Aut((G^{[\alpha]})^0)$. Then, $(\Fix y_{\ast})^0=S^0$. Write $n$ for the 
order of $y_{\ast}$, and $z_{\ast}=I+y_{\ast}+\cdots+y_{\ast}^{n-1}$. Then, $z_{\ast}$ is an endomorphism 
of $(G^{[\alpha]})^0$ and $G_{[\alpha]}$ is equal to the neutral subgroup of $$\{x\in(G^{[\alpha]})^0: 
z_{\ast}(x)=1\}.$$ Thus, $G_{[\alpha]}$ is a closed subgroup of $G$. As $\frg_{[\alpha]}$ is abelian, so 
is $G_{[\alpha]}$. Since $\frg_{[\alpha]}\cap\frs=0$, we know that $G_{[\alpha]}\cap S$ is a finite group.  
\end{proof}

For any $\xi\in\Hom(S/\ker\alpha, G_{[\alpha]}\cap S)$, define $s_{\alpha,\xi}: S\rightarrow S$ by 
$$s_{\alpha,\xi}(x)=x\xi(x),\ \forall x\in S.$$ Defined in this way, $s_{\alpha,\xi}$ is a self-map 
of $S$. 

\begin{lemma}\label{L:transvection}
For any $\xi\in\Hom(S/\ker\alpha, G_{[\alpha]}\cap S)$, there exists $$n\in G_{[\alpha]}\cap N_{G}(S)$$ 
such that $$\Ad(n)|_{S}=s_{\alpha,\xi}.$$ Particularly, $s_{\alpha,\xi}$ is an automorphism of $S$. 
\end{lemma}

\begin{proof}
As $\Ad(y^{-1})-I$ is an endomorphism of $\frg_{[\alpha]}$ without eigenvalue $0$ due to $Z_{\frg_{[\alpha]}}
(y)=0$, we know that $$\Ad(y^{-1})-I: G_{[\alpha]}\rightarrow G_{[\alpha]}$$ is surjective. Hence, there 
exists $n\in G_{[\alpha]}$ such that $y^{-1}nyn^{-1}=\xi(y)$. Thus, $nyn^{-1}=y\xi(y)$. Since $\ker\alpha$ 
commutes with $G_{[\alpha]}$ as $\ker\alpha$ acts trivially on $\frg_{[\alpha]}$, we have $nxn^{-1}=x$ for 
any $x\in\ker\alpha$. Thus, $$n\in G_{[\alpha]}\cap N_{G}(S)$$ and $$\Ad(n)|_{S}=s_{\alpha,\xi}.$$ 
\end{proof}

Write $$R^{\vee}(\alpha)=\Hom(S/\ker\alpha, G_{[\alpha]}\cap S),$$ and call it the root transvection group. 
We call $s_{\alpha,\xi}$ ($\xi\in R^{\vee}(\alpha)$) a transvection. 

\begin{definition}\label{D:smallWeyl}
Let $W_{small}(G,S)$ denote the subgroup of $N_{G^{0}}(S)/Z_{G^{0}}(S)$ generated by $s_{\alpha}$ (for 
infinite roots $\alpha$) and transvections $s_{\alpha,\xi}$ (for finite roots $\alpha$ and $\xi\in 
R^{\vee}(\alpha)$). We call $W_{small}(G,S)$ the small Weyl group of $S$ in $G$.
\end{definition} 

By Lemma \ref{L:transvection}, $W_{small}(G,S)$ is a subgroup of $N_{G^{0}}(S)/Z_{G^{0}}(S)$. We also note 
that $N_{G^{0}}(S)/Z_{G^{0}}(S)$ is a quotient group of the Weyl group $N_{G^{0}}(S)/S^{0}$. The action of 
$W_{small}(G,S)$ on $S$ induces an action on the character group $X^{\ast}(S)$.

\begin{proposition}\label{P:affineRD}
The set $R(G,S)$ is stable under the action of $W_{small}(G,S)$, the map $\alpha\mapsto\check{\alpha}$ and 
the root transvection groups $R^{\vee}(\alpha)$ are permuted under the action of $W_{small}(G,S)$. 
\end{proposition}

\begin{proof}
This follows from Lemmas \ref{L:reflection} and \ref{L:transvection}.  
\end{proof}

\begin{definition}\label{D:affineRD}
We call $R(G,S)$ together with the coroots $\alpha\mapsto\check{\alpha}$ for infinite roots and the root 
transvection groups $R^{\vee}(\alpha)$ for finite roots the {\it affine root datum} for $G$ with respect 
to $S$.  
\end{definition} 

By Proposition \ref{P:affineRD}, one can show that $(R(G,S), W_{small}(G,S))$ inherits many structures and 
properties like the classical root system. A related structure of affine root datum is studied in \cite{Reeder-Yu}.
Different with that in \cite{Reeder-Yu}, we emply ideas in \cite{Han-Vogan} to consider all finite roots and 
their root transvection groups. A more general theory of {\it twisted root datum} and a more complete investigation 
of the structures are given in \cite{Yu-TRD}.  

Write $$R_{0}=R_{0}(G,S)$$ for the set of finite roots. Write $$R'=R'(G,S)=\{\alpha|_{S^{0}}:\alpha
\in R(G,S)-R_{0}(G,S)\}.$$ 

\begin{proposition}\label{P:R'}
$R'(G,S)$ in a root system in $X^{\ast}(S^{0})$ in the sense of \cite[Definition 2.2]{Yu-dimension}.  
\end{proposition}

\begin{proof}
This follows from Proposition \ref{P:affineRD}.  
\end{proof}

For the application in the study of dimension datum, we see in Subsection \ref{SS:Weyl} that only infinite 
roots and their corresponding reflections matter. We call infinite roots $$R(G,S)-R_{0}(G,S)$$ together with 
the coroots $\alpha\mapsto\check{\alpha}$ ($\alpha\in R(G,S)-R_{0}(G,S)$) the {\it affine root system} for 
$G$ with respect to $S$.

\subsection{Weyl product}\label{SS:Weyl}


Write $$T=Z_{G^{0}}(S^{0}).$$ The following fact is well known. 

\begin{lemma}\label{L:T}
$T$ is a maximal torus of $G^{0}$. 
\end{lemma}

\begin{proof}
Choose $y\in S$ generates $S/S^{0}$. Then, $\mathfrak{t}^{y}=\mathfrak{s}$. Apparently, $\mathfrak{s}\subset 
z(\mathfrak{t})$. Thus, $\mathfrak{t}^{y}\subset z(\mathfrak{t})$. Write $\mathfrak{t}=z(\mathfrak{t})\oplus
\mathfrak{t}^{s}$ with $\mathfrak{t}^{s}=[\mathfrak{t},\mathfrak{t}]$ the derived subalgebra. Then, 
$\mathfrak{t}^{y}\subset z(\mathfrak{t})$ implies $(\mathfrak{t}^{s})^{y}=0$. By a theorem of Borel, this 
indicates that $\mathfrak{t}^{s}$ is abelian. Thus, $\mathfrak{t}$ is abelin. We know $T$ is connected and it 
contains any maximal torus which contains $S^{0}$. Therefore, $T$ is a maximal torus.  
\end{proof}

\begin{lemma}\label{L:split}
There is an exact sequence $$1\rightarrow N_{T}(S)/S^{0}\rightarrow W(G,S)\rightarrow W_{R'}\rightarrow 1.$$
\end{lemma}

\begin{proof}
Retricting the action of elements in $W(G,S)$ to $S^{0}$, it induces an action of $W(G,S)$ on $R'$. Thus, 
we have a natural homomorphism $$\phi: W(G,S')\rightarrow \Aut(R').$$ Apparently the image is equal to 
$W_{R'}$. For any $gS^{0}\in W(G,S')$ ($g\in G^{0}$), $gS^{0}\in\ker\phi$ if and only $\alpha(gxg^{-1})=
\alpha(x)$ for any $x\in S^{0}$ and any infinite root $\alpha$. Since $\alpha(gxg^{-1})=\alpha(x)=1$ for 
any $x\in S^{0}$ and any finite root $\alpha$. Then, $gS^{0}\in\ker\phi$ if and only $\alpha(gxg^{-1})=
\alpha(x)$ for all roots $\alpha$. Which is equivalent to $gxg^{-1}x^{-1}\in Z(G^{0})$ for any $x\in S^{0}$. 
Then, $$gxg^{-1}x^{-1}\in Z(G^{0})\cap [G^{0},G^{0}]=Z((G^{0})_{der})$$ for any $x\in S^{0}$. As 
$Z((G^{0})_{der})$ is a finite group, we get $gxg^{-1}x^{-1}=1$ for any $x\in S$ by continuity. Hence, 
$g\in T=Z_{G^{0}}(S^{0})$. This proves the conclusion. 
\end{proof}

\begin{lemma}\label{L:finiteRoot}
For any $s\in S'$, $$\prod_{\alpha\in R_{0}}(1-\alpha(s))=|N_{T}(S)/S^{0}|.$$
\end{lemma}
\begin{proof}
The group $N_{T}(S)$ consists of elements $t$ in $T$ such that $tst^{-1}s^{-1}\in S^{0}$. Considering the 
quotient group $T/S^{0}$ and the induced action by $s$, then $$N_{T}(S)/S^{0}\cong\{tS^{0}\in T/S^{0}: 
s(tS^{0})s^{-1}(tS^{0})^{-1}=S^{0}\}.$$ The latter is just the fixed point group of $s$ in $T/S^{0}$. 
The action of $s$ on $T/S^{0}$ is given by a matrix $X\in\GL(k,\mathbb{Z}),$ where $k=\dim(T/S^{0})$. 
Write $f(u)$ for the characteristice polynomial of $X$. Then, $$(T/S^{0})^{s}\cong (I-X)^{-1}
(\mathbb{Z}^{k})/\mathbb{Z}^{k}$$ and $$f(u)=\prod_{\alpha\in R_{0}}(u-\alpha(s)).$$ Thus, 
$$|(T/S^{0})^{s}|=|\det(I-X)|=|f(1)|=\prod_{\alpha\in R_{0}}(1-\alpha(s)).$$ Together with the 
identification $N_{T}(S)/S^{0}\cong(T/S^{0})^{s}$ shown in the above, we get $$\prod_{\alpha\in R_{0}}
(1-\alpha(s))=|N_{T}(S)/S^{0}|.$$ 
\end{proof}



For each $\alpha'\in R'$ with $\frac{1}{2}\alpha'\not\in R'$, set $$R_{1,\alpha'}=\{\beta\in R(G,A):
\beta|_{S^{0}}=\alpha'\},\quad R_{2,\alpha'}=\{\beta\in R(G,A):\beta|_{S^{0}}=2\alpha'\}$$ and 
$$R_{\alpha'}=R_{1,\alpha'}\cup R_{2,\alpha'}.$$ Write $$m_{1,\alpha'}=|R_{1,\alpha'}|,\quad 
m_{2,\alpha'}=|R_{2,\alpha'}|$$ and $$m_{\alpha'}=m_{1,\alpha'}+2m_{2,\alpha'}.$$ 

\begin{lemma}\label{L:density}
For any $\alpha\in R_{1,\alpha'}$, \begin{equation}\label{Eq:density1}
\prod_{\beta\in R_{\alpha'}}(1-\beta(s))=1-\alpha(s)^{m_{\alpha'}}.\end{equation}
\end{lemma}

\begin{proof}
Write $$S_{\alpha'}=\ker(\alpha')=\ker\alpha\cap S^{0}.$$ Set $$H_{\alpha'}=Z_{G}(S_{\alpha'}).$$ Then, 
all root spaces $\mathfrak{g}_{\beta}$ ($\beta\in R_{\alpha'}$) are contained in $\mathfrak{h}_{\alpha'}$. 
We may assume that $G=H_{\alpha'}$, i.e., all infinite roots are in the set $R_{\alpha'}$. Write 
$\mathfrak{h}_{0}=[\mathfrak{g}_{0},\mathfrak{g}_{0}]$. Let $$\phi: G\rightarrow\Aut(\frh_{0})$$ be 
the natural homomorphism from the adjoint action of $G$ on $\frh_0$. Write $H=\Aut(\frh_0)$. We have 
$\ker\phi=Z_{G}([G^{0},G^{0}])$. As $\frg_{\beta}\subset[\mathfrak{g}_{0},\mathfrak{g}_{0}]$ for any 
$\beta\in R_{\alpha'}$, we have $\beta|_{S\cap\ker\phi}=1$ for any $\beta\in R_{\alpha'}$. Hence, it 
is equivalent to consider $\phi(S)\subset H$, which is a generalized Cartan subgroup of dimension one. 

Up to isomorphism, the pair $(\phi(S),\mathfrak{h}_{0})$ has only two possibilities: (1), $\fru_0=
\underbrace{\fru_1\oplus\cdots\oplus\fru_1}_{m}$, $\phi(S)=\langle\Delta(S_1),\theta\rangle$, where 
$\fru_1\cong\mathfrak{su}(2)$, $S_1$ is a maximal torus of $\Int(\fru_1)$, $$\Delta:\Int(\fru_1)
\rightarrow\prod_{1\leq i\leq m}\Int(\fru_1)$$ is the diagonal map, and $$\theta(X_1,\dots,X_{m})=
(X_2,\dots,X_{m},X_1).$$ (2), $\fru_0=\underbrace{\fru_1\oplus\cdots\oplus\fru_1}_{m}$, $\phi(S)=
\langle\Delta(S_1),\theta\rangle$, where $\fru_1\cong\mathfrak{su}(3)$, $\sigma\in\Aut(\fru_1)$ such 
that $\fru_1^{\sigma}\cong\mathfrak{so}(3)$ and hence $\Int(\fru_1)^{\sigma}\cong\SO(3)$, $S_1$ is a 
maximal torus of $\Int(\fru_1)^{\sigma}$, $$\Delta:\Int(\fru_1)\rightarrow\prod_{1\leq i\leq m}
\Int(\fru_1)$$ is the diagonal map, and $$\theta(X_1,\dots,X_{m})=(X_2,\dots,X_{m},\sigma(X_1)).$$ 

In case (1), define $\beta_{0}\in\Hom(S,\U(1))$ by $\beta_{0}|_{S^{0}}=1$ and $\beta_{0}(\theta)=
e^{\frac{2\pi i}{m}}$. Then, $$R_{\alpha'}=\{\alpha+k\beta_{0}: 0\leq k\leq m-1\}.$$ Thus, 
$m_{1,\alpha'}=m$, $m_{2,\alpha'}=0$, and $m_{\alpha'}=m$. 

In case (2), define $\beta_{0}\in\Hom(S,\U(1))$ by $\beta_{0}|_{S^{0}}=1$ and $\beta_{0}(\theta)=
e^{\frac{\pi i}{m}}$. Then, $$R_{\alpha'}=\{\alpha+k\beta_{0}: 0\leq k\leq 2m-1\}\cup\{2\alpha+
(2k+1)\beta_{0}: 0\leq k\leq m-1\}.$$ Thus, $m_{1,\alpha'}=2m$, $m_{2,\alpha'}=m$, and $m_{\alpha'}
=4m$. In either case, the conclusion follows from the elementary formula $$\prod_{0\leq k\leq l-1}
(1-e^{\frac{2k\pi i}{l}}t)=1-t^{l}$$ for any $l\geq 1$.  
\end{proof}

\begin{proposition}\label{P:m-alpha}
Let $\alpha'\in R$ with $\frac{1}{2}\alpha'\not\in R$. For any two roots $\alpha,\tilde{\alpha}
\in R_{1,\alpha'}$, $[m_{1,\alpha'}\alpha]=[m_{1,\alpha'}\tilde{\alpha}]$, and $m_{\alpha'}=
m_{1,\alpha'}$ or $2m_{1,\alpha'}$. 
\end{proposition}

\begin{proof}
In the proof of Lemma \ref{L:density}, we have described the structure of the set $R_{\alpha'}$. 
From it we get the conclusion of this proposition. 
\end{proof}

\medskip 

Choose a simple system $\{\alpha'_{1},\cdots,\alpha'_{t}\}$ of the root system $R'$ by choosing a simple 
system for each irreducible factor of $R'$. In case some simple factor of $R'$ is not reduced, i.e, it 
is of type $\BC_{l}$ for some $l\geq 1$, a simple system means simple system for the corresponding root 
system $\B_{l}$. For each $i$, choose $\alpha_{i}\in R(G,S)$ such that $$\alpha_{i}|_{S^{0}}=\alpha'_{i}.$$ 
Since the characters $\alpha'_{1},\cdots,\alpha'_{t}$ are linearly independent, there exists $s_{0}\in S'$ 
such that $$\alpha_{1}(s_0)=\cdots=\alpha_{t}(s_0)=1.$$ Let $R_{s_0}$ be the sub-root system of $R(G,S)$ 
generated by $\alpha_1,\cdots,\alpha_{t}$. Then, the map $$R_{s_0}\rightarrow R',\quad\alpha\mapsto
\alpha|_{S^{0}}$$ is an isomorphism of root systems. For each $\alpha\in R_{s_0}$, we write $\alpha'=
\alpha|_{S^{0}}$ to denote this bijection. 

Choose a positive system $R'^{+}$ in $R'$. Set\footnotemark $$\delta=\delta_{R}=\frac{1}{2}\sum_{\alpha'
\in R'^{+},\textrm{reduced}}m_{\alpha'}\alpha.$$ In the above the summation takes over positive roots 
$\alpha'$ in $R'$ with $\frac{1}{2}\alpha'\not\in R'$; for any such $\alpha'$, we choose a root $\alpha
\in R_{1,\alpha'}$.
\footnotetext{Precisely to say, $2\delta$ is a well-defined character of $S$, but $\delta$ may be not. 
While we use $\delta$ below, it appears in the form of $\delta-w\delta$ ($w\in W_{small}(G,S)$), which 
are all well-defined characters of $S$. Thus, the express of $\delta$ as above arises no ambiguity.} 

\begin{definition}\label{D:AR}
Define $$A_{R}=\frac{1}{|W_{R'}|}\sum_{w\in W_{R_{s_0}}}\epsilon(w)[\delta_{R}-w\delta_{R}].$$  
\end{definition} 

\begin{lemma}\label{L:AR}
$2\delta_{R}$ and $\delta_{R}-w\delta_{R}$ ($W\in W_{R_{s_0}}$) are in $X^{\ast}(S)$ and they are 
independent of the choice of $\alpha$ for $\alpha'\in R'$.  

$A_{R}\in\mathbb{Q}[X^{\ast}(S)]$ depends only on $R$, not on the choice of $\alpha_1,\dots,\alpha_{t}$ 
(or of $s_0$).
\end{lemma}

\begin{proof}
By definition, $2\delta_{R}$ is an integral combination of $\{m_{\alpha'}\alpha: \alpha'\in R'\}$. 
Thus, it is in $X^{\ast}(S)$. From basic properties of rot system and Weyl group, each term $\delta_{R}-
w\delta_{R}$ ($W\in W_{R_{s_0}}$) is an integral combination of $\{m_{\alpha'}\alpha: \alpha'\in R'\}$. 
Thus, they are in $X^{\ast}(S)$. By Proposition \ref{P:m-alpha}, $[m_{\alpha'}\alpha]=[m_{\alpha'}
\tilde{\alpha}]$ for any $\alpha'\in R'$ and any $\alpha,\tilde{\alpha}\in R_{1,\alpha'}$. Thus, 
$2\delta_{R}$ and $\delta_{R}-w\delta_{R}$ ($W\in W_{R_{s_0}}$) are independent of the choice of 
$\alpha$ for $\alpha'\in R'$.  

By the first statement shown above and the definition of $A_{R}$, we have $$A_{R}\in\mathbb{Q}
[X^{\ast}(S)]$$ and it is independent of the choice of $\alpha_1,\dots,\alpha_{t}$ (or of $s_0$).
\end{proof}


\begin{proposition}\label{P:density2}
For any $s\in S'$, $$D(s)=\frac{1}{|W_{R'}|}\sum_{\tau\in W_{s_{0}}}\tau\cdot A_{R}(s).$$  
\end{proposition}

\begin{proof}
When $S$ is connected, this is shown in the proof of \cite[Proposition 3.7]{Yu-dimension}. In general, 
Lemmas \ref{L:split}, \ref{L:finiteRoot}, \ref{L:density} reduce it to the connected case.  
\end{proof}


\begin{proposition}\label{P:S-S'}
For any character $F\in\mathbb{Q}[S]$, $F|_{S'}=0$ if and only if $F=0$. 
\end{proposition}
\begin{proof}
Write $m=|S/S^{0}|$ and $\omega_{m}=e^{\frac{2\pi i}{m}}$. Choose $s\in S$ generates $S/S^{0}$. Define 
$\beta_{0}\in X^{\ast}(S)$ by $\beta_{0}|_{S^{0}}=1$ and $\beta_{0}(s)=\omega_{m}$. For any two linear 
characters $\lambda,\lambda'\in X^{\ast}(S)$, $\lambda|_{S^{0}}=\lambda'|_{S^{0}}$ if and only if $\lambda'
=\lambda+k\beta_{0}$ for some $k$, $0\leq k\leq m-1$. The conclusion follows from the linear independence 
of characters of $S^{0}$ and the $\mathbb{Q}$-independence of $1,\omega_{m},\dots,\omega_{m}^{m-1}$. 
\end{proof}

By Proposition \ref{P:S-S'}, any character $F\in\mathbb{Q}[S]$ is determined by its restriction (evaluation) 
on $S'$. Particularlry, if $F|_{S'}=0$, then $F|_{S^{0}}$.

\subsection{Irreducible affine root systems}\label{SS:irreducibleARS}

\begin{definition}
We say two pairs  $(G_1,S_1)$ and $(G_2,S_2)$ of generalized Cartan subgroups in compact Lie groups 
{\it isogenous} if there is another pair $(G_3,S_3)$ and homomorphisms $f_{i}: G_{i}\rightarrow G_{3}$ 
($i=1,2$) such that $f_1,f_2$ are isomorphisms on Lie algebras, and $f_{1}(S_1)=f_{2}(S_2)=S_3$. 

We call a pair $(G,S)$ {\it irreducible} if there is are no pairs $(G_1,S_1)$ and $(G_2,S_2)$ with $\dim G_1
>0$ and $\dim G_2>0$, such that $(G,S)$ is isogenous to $(G_{1}\times G_{2},S_{1}\times S_{2})$. 
\end{definition}

Now we describe all possible pairs $(G,S)$ up to isogeny, and give the multiples $m_{1,\alpha'},m_{\alpha'}$. 
For an irreducible pair $(G,S)$, either $G^{0}$ is a torus, or $\frg_0$ is the product of several copies 
of a simple compact Lie algebra and a generator of $S$ acts on $\frg_0$ by permuting its simple factors 
transitively. In the first case, either $S^{0}=G^{0}$ is a one-dimensional torus, or $\dim G^{0}=\phi(n)$ 
($\phi(n)$ means Euler function) and $S=\langle\theta\rangle$ with $\theta_{\ast}=\Ad(\theta)\in
\GL(X^{\ast}(G^{0}))$ having degree $n$ cyclotomic polynomial as its characteristic polynomial. 
There are no infinite roots in the first case. 

In the latter case, write $\Phi$ for affine root system of $G$, and $R'=R'(G,S)$. Then, the pair $(\Phi,R')$ 
determines $(G,S)$ up to isogeny. All possible pairs $(\Phi,R')$ are as follows. We also indicate the 
multiples $m_{1,\alpha'}$ and $m_{\alpha'}$. \begin{enumerate}
\item $(\Phi,R')=(m\Phi_0,\Phi_0)$, where $\Phi_0$ is an irreducible root system, $m\geq 1$;  
$m_{\alpha'}=m_{1,\alpha'}=m$ for any $\alpha'\in R'$.  
\item $(\Phi,R')=(m\A_{2n},\BC_{n})$ ($n\geq 1$); $m_{1,\alpha'}=2$ and $m_{\alpha'}=4$ for any 
short root $\alpha'\in R'$; $m_{\alpha'}=m_{1,\alpha'}=2$ for middle root $\alpha'\in R'$. 
\item $(\Phi,R')=(m\A_{2n-1},\C_{n})$ ($n\geq 2$); $m_{\alpha'}=m_{1,\alpha'}=2$ for any short 
root $\alpha'\in R'$; $m_{\alpha'}=m_{1,\alpha'}=1$ for long root $\alpha'\in R'$. 
\item $(\Phi,R')=(m\D_{n},\B_{n-1})$ ($n\geq 4$); $m_{\alpha'}=m_{1,\alpha'}=2$ for any short 
root $\alpha'\in R'$; $m_{\alpha'}=m_{1,\alpha'}=1$ for long root $\alpha'\in R'$.
\item $(\Phi,R')=(m\D_{4},\G_{2})$; $m_{\alpha'}=m_{1,\alpha'}=3$ for any short root $\alpha'\in 
R'$; $m_{\alpha'}=m_{1,\alpha'}=1$ for long root $\alpha'\in R'$.
\item $(\Phi,R')=(m\E_{6},\F_{4})$; $m_{\alpha'}=m_{1,\alpha'}=2$ for any short root $\alpha'\in 
R'$; $m_{\alpha'}=m_{1,\alpha'}=1$ for long root $\alpha'\in R'$. 
\end{enumerate}

\subsection{Further about the character $A_{R}$}\label{SS:AR} 

For a given $R'$, the character $A_{R}$ is determined by the multiples $\{m_{\alpha'}:\alpha'\in R'\}$ 
and a set of lifting $\{\alpha_{1},\dots,\alpha_{t}\}$ of simple roots of $R'$. The latter is determined 
by the base point $s_0$. 

For a given base point $s_0\in S'$, the map $s(\in S^{0})\mapsto ss_0$ identifies $S'$ with $S^{0}$. 
The description of $(R,R')$ reduces to irreducible case. By the description in the irreducible case as 
given above, we see that $\{m_{\alpha'}\alpha':\alpha'\in R'\}$ looks like a root system. Compared 
with \cite[Definition 2.2]{Yu-dimension}, it satisfies the first condition (Weyl group permutation), 
but does not satisfy the second condition (strong integrality) in general. 

\smallskip 

Write $m=|S/S^{0}|$. Choose and fix a point $\tilde{s}\in S'$ such that $o(\tilde{s})=m$. Then, $$S=
S^{0}\times\langle\tilde{s}\rangle.$$ Define $\beta_{0}\in X^{\ast}(S)$ by $\beta_{0}|_{S^{0}}=1$ and 
$\beta_{0}(s)=e^{\frac{2\pi i}{m}}$. For each $\alpha'\in R'$, let $\alpha\in X^{\ast}(S)$ be defined 
by $\alpha|_{S^{0}}=\alpha'$ and $\alpha(\tilde{s})=1$. Then $$\{\alpha+k\beta_{0}:0\leq k\leq m-1\}.$$ 
are all liftings of $\alpha'$ in $X^{\ast}(S)$.  

\begin{lemma}\label{L:m1}
For any root $\alpha'\in R'$ with $\frac{1}{2}\alpha'\not\in R'$, we have $m_{1,\alpha'}|m$.  
\end{lemma}
\begin{proof}
Choose one $\alpha\in R_{1,\alpha'}$. From the description of $R_{1,\alpha'}$ given in the proof of Lemma 
\ref{L:density}, we see that there exists a degree $m_{1,\alpha'}$ character $\beta\in X^{\ast}(S/S^{0})$ 
such that $$R_{1,\alpha'}=\{\alpha+k\beta: 0\leq k\leq m_{1,\alpha'}\}.$$ Thus, $$m_{1,\alpha'}||S/S^{0}|
=m.$$
\end{proof}

Write $$R'=\bigsqcup_{1\leq j\leq k} R'_{i}$$ for the decomposition of $R'$ into irreducible root systems. 
For each $i$, write $m_{i}=m_{1,\alpha'}$ for a short root in $R'_{i}$. From the description of 
$(R-R_{0},R')$ given in Subsection \ref{SS:irreducibleARS}, we see that there are three possibilities. 
\begin{enumerate}
\item $R'_{i}$ is a reduced irreducible root system, all $m_{\alpha'}$ ($\alpha'\in R_{i}$) equal to 
$m_{i}$. In this case $m_{i}|m$ by Lemma \ref{L:m1}. 
\item $R'_{i}$ is a reduced non-simply-laced irreducible root system, and $m_{\alpha'}=n_{\alpha'}^{-1}
m_{i}$, where $n_{\alpha'}$ is the square of the ratio of the length of $\alpha'$ and of short roots in 
$R'_{i}$. In this case $m_{i}|m$ by Lemma \ref{L:m1}. 
\item $R'_{i}$ is a non-reduced irreducible root system, for a short root $\alpha'\in R'_{i}$, 
$m_{1,\alpha'}=m_{i}$, $m_{1,\alpha'}=2m_{i}$; for a middle length root $\alpha'\in R'_{i}$, $m_{1,\alpha'}
=m_{\alpha'}=m_{i}$. In this case $m_{i}|m$ by Lemma \ref{L:m1}. Moreove, $m_{i}$ is even in this case 
from the description of $(R,R')$ given in Subsection \ref{SS:irreducibleARS}. 
\end{enumerate}

\begin{remark}\label{R:R-AR}
In the above, we have described the possibilities of the multiples $m_{\alpha}$ and and of the liftings 
$\{\alpha_1,\dots,\alpha_{t}\}$. From that, we could describe the $R-R_{0}$ and calculate $A_{R}$ from 
the root system $R'$, the multiples $m_{\alpha}$ and and of the liftings $\{\alpha_1,\dots,\alpha_{t}\}$. 
\end{remark}

\subsection{Dimension data of disconnected subgroups}\label{SS:disconnected}

We call a closed abelian subgroup $S$ of $G$ a {\it generalized torus} if $S/S^{0}$ is a cyclic group.
Write $$\Gamma^{0}=N_{G^{0}}(S)/Z_{G^{0}}(S).$$ Fix a connected component $S'$ of $S$ such that $S'$ 
generates $S$. Then, $S'$ is a closed {\it commutative connected subset} in $G$.    

Choose a biinvariant Riemannian metric on $G$. By restriction it gives a $\Gamma^{0}$ invariant positive 
definite inner product on the Lie algebra of $S^{0}$, hence also a $\Gamma^{0}$ invariant positive 
definite inner product on $X^{\ast}(S^{0})$, denoted by $(\cdot,\cdot)$. Set $$\Psi_{S^{0}}=\{0\neq\alpha
\in X^{\ast}(S^{0}):\frac{2(\lambda,\alpha)}{(\alpha,\alpha)}\in\mathbb{Z},\forall\lambda\in X^{\ast}
(S^{0})\}.$$ It is a root system in $X^{\ast}(S^{0})$ (\cite[Definition 2.2]{Yu-dimension}). Choose and 
fix a positive system $\Psi_{S^{0}}^{+}$ of $\Psi_{S^{0}}$. 

Set $$\Psi_{S}=\{\alpha\in X^{\ast}(S): \alpha|_{S^{0}}\in\Psi_{S^{0}}\}.$$ For any 
$\alpha\in\Psi_{S}$, there exists a unique $\check{\alpha}\in X_{\ast}(S^{0})$ such that $$\lambda
(\check{\alpha})=\frac{2(\lambda,\alpha|_{S^{0}})}{(\alpha|_{S^{0}},\alpha|_{S^{0}})},\forall\lambda
\in X^{\ast}(S^{0}).$$ Define $s_{\alpha}: X^{\ast}(S)\rightarrow X^{\ast}(S)$ by $$s_{\alpha}(\lambda)
=\lambda-(\lambda|_{S^{0}})(\check{\alpha})\alpha,\ \forall\lambda\in X^{\ast}(S).$$ Then, $s_{\alpha}
(\alpha)=-\alpha$, $s_{\alpha}^{2}=1$, and $$s_{\alpha}(\lambda)=\lambda\Leftrightarrow(\lambda|_{S^{0}})
(\check{\alpha})=0.$$ 


\begin{definition}\label{D:affineRS}
We call a subset $R$ of $X^{\ast}(S)$ an affine root system on $S$ if it satisfies the following 
conditions\footnotemark, 
\begin{enumerate}
\item $R\subset\Psi_{S}$.  
\item $s_{\alpha}(R)=R$ for any $\alpha\in R$. 
\item for any $\alpha\in R$, $2\alpha\not\in R$. 
\item for any $\alpha\in R$ with $2\alpha|_{S^{0}}\not\in\{\beta|_{S^{0}}:\beta\in R\}$, there exists 
$\delta\in X^{\ast}(S/S^{0})$ ($\subset X^{\ast}(S)$) of order $m$ ($\in\mathbb{Z}_{\geq 1}$) such that 
\begin{eqnarray*} 
&&\{\beta\in R:\beta|_{S^{0}}=\alpha|_{S^{0}}\}\\&=&\{\alpha+k\delta: 0\leq k\leq m-1\}.\end{eqnarray*} 
\item for any $\alpha\in R$ with $2\alpha|_{S^{0}}\in\{\beta|_{S^{0}}:\beta\in R\}$, there exists 
$\delta\in X^{\ast}(S/S^{0})$ ($\subset X^{\ast}(S)$) of order $2m$ ($m\in\mathbb{Z}_{\geq 1}$) such that  
\begin{eqnarray*}
&&\{\beta\in R:\beta|_{S^{0}}\in\mathbb{Z}\cdot\alpha|_{S^{0}}\}\\&=&\{\alpha+k\delta: 0\leq k\leq 2m-1\}
\cup\{2\alpha+(2k+1)\delta: 0\leq k\leq m-1\}.\end{eqnarray*}
\end{enumerate}
\end{definition}

\footnotetext{Note that $S/S^{0}$ is a finite cyclic group. Actually one can show that the axioms (1)-(3) imply 
the axioms (4) and (5) (cf. \cite{Yu-TRD}).} 

\begin{lemma}\label{L:affineRS}
Let $H$ be a compact subgroup of $G$ with $S$ a generalized Cartan subgroup of $H$. Then the affine root system 
of $H$ with respect to $S$ satisfies the Definition \ref{D:affineRS}.  
\end{lemma}

\begin{proof}
This follows from Lemma \ref{L:infinite1}, Lemma \ref{L:reflection}, and the proof of Lemma \ref{L:density} 
(or from the description of the set $(R-R_{0},R')$ given in Subsection \ref{SS:irreducibleARS}).
\end{proof}

\begin{definition}\label{D:Weyl-R'}
For an affine root system $R$ on $S$, write $W_{R}$($\subset\Aut(S)$) for the finite group generated by 
$s_{\alpha}$ ($\alpha\in R$), and call it the Weyl group of $R$. 

Write $R'=\{\alpha|_{S^{0}}: \alpha\in R\}.$
\end{definition}

By Definition \ref{D:affineRS}, it is clear that $R'$ is a root system in $X^{\ast}(S^{0})$ in the sense of 
\cite[Definition 2.2]{Yu-dimension}.    

\smallskip 

Choose a simple system $\{\alpha'_{1},\cdots,\alpha'_{t}\}$ of the root system $R'$. For each $i$, 
choose $\alpha_{i}\in R$ such that $\alpha_{i}|_{S^{0}}=\alpha'_{i}.$ Take $s_{0}\in S'$ such that 
$$\alpha_{1}(s_0)=\cdots=\alpha_{t}(s_0)=1.$$ Let $$R_{s_0}=\{\alpha\in R:\alpha(s_0)=1\}.$$ Then, 
the map $$R_{s_0}\rightarrow R',\quad\alpha\mapsto\alpha'=\alpha|_{S^{0}}$$ is an isomorphism of 
root systems. 

For each $\alpha'\in R'$ with $\frac{1}{2}\alpha'\not\in R'$, set $$R_{1,\alpha'}=\{\beta\in R:
\beta|_{S^{0}}=\alpha'\},\quad R_{2,\alpha'}=\{\beta\in R(G,A):\beta|_{S^{0}}=2\alpha'\}$$ and 
$$R_{\alpha'}=R_{1,\alpha'}\cup R_{2,\alpha'}.$$ Write $$m_{1,\alpha'}=|R_{1,\alpha'}|,\quad 
m_{2,\alpha'}=|R_{2,\alpha'}|$$ and $m_{\alpha'}=m_{1,\alpha'}+2m_{2,\alpha'}$. 

Write $R'=R\cap\Psi_{S^{0}}^{+}$. 

\begin{definition}\label{D:delta-AR}
Set $$\delta=\delta_{R}=\frac{1}{2}\sum_{\alpha'\in R'^{+},\textrm{reduced}}m_{\alpha'}\alpha.$$ 

Define $$A_{R}=\frac{1}{|W_{R'}|}\sum_{w\in W_{R_{s_0}}}\epsilon(w)[\delta_{R}-w\delta_{R}].$$  
\end{definition}

The following lemma could be shown the same as for Lemma \ref{L:AR}. 

\begin{lemma}\label{L:AR2}
$2\delta_{R}$ and $\delta_{R}-w\delta_{R}$ ($W\in W_{R_{s_0}}$) are in $X^{\ast}(S)$ and they are 
independent of the choice of $\alpha$ for $\alpha'\in R'$.  

$A_{R}\in\mathbb{Q}[X^{\ast}(S)]$ depends only on $R$, not on the choice of $\alpha_1,\dots,\alpha_{t}$ 
(or of $s_0$).
\end{lemma}

\begin{lemma}\label{L:AR3}
We have $$A_{R}=\frac{1}{|W_{R}|}\sum_{w\in W_{R}}\epsilon(w)[\delta_{R}-w\delta_{R}].$$
\end{lemma}

\begin{proof}
Write $R^{+}=\{\beta\in R:\beta|_{S^{0}}\in R'^{+}\}$. Due to $$\delta_{R}-w\delta_{R}=
\sum_{\beta\in R^{+}\cap-w\cdot R^{+}}\beta,$$ we see that $\delta_{R}-w\delta_{R}$ ($w\in W_{R}$) depends 
only on the image of projection of $w$ in $W_{R'}$. Thus, $$A_{R}=\frac{1}{|W_{R}|}\sum_{w\in W_{R}}
\epsilon(w)[\delta_{R}-w\delta_{R}].$$  
\end{proof} 

\begin{definition}\label{D:character2}
For an affine root system $R$ on $S$, and a finite group $W$ between $W_{R}$ and $\Aut(S)$, set 
\[F_{R,W}=\frac{1}{|W|}\sum_{\gamma\in W}\gamma\cdot A_{R}.\] 
\end{definition}

The following proposition is a generalization of \cite[Proposition 3.8]{Yu-dimension}. 
\begin{proposition}\label{P:compact-root system}
Given a compact Lie group $G$, let $H_{1},H_2,\dots,H_{s}\subset G$ ($s\geq 2$) be a collection 
of closed subgroups of $G$. For non-zero constants $c_1,\cdots,c_{s}$, 
\[\sum_{1\leq i\leq s} c_{i}\mathscr{D}_{H_i}=0\] if and only if: for any closed commutative 
connected subset $S'$ of $G$, \[\sum_{1\leq j\leq t}\frac{c_{i_{j}}}{|G_{i_{j}}/G_{i_{j}}^{0}|}
F_{\Phi_{j},\Gamma^{\circ}}|_{S'}=0.\] Here $$\Gamma^{\circ}=N_{G}(S')/Z_{G}(S'),$$ 
$\{H_{i_{j}}|i_1\leq i_2\leq \cdots\leq i_{t}\}$ are all subgroups amongst $\{H_{i}|1\leq i\leq s\}$ 
with $H_{i_{j}}$ contains a maximal commutative connected subset conjugate to $S'$ and $\Phi_{j}$ 
is the affine root system of $H_{i_{j}}$ with respect to $S=s^{-1}S'$ ($s\in S'$). If a subgroup 
$H_{i}$ contains $k$ maximal commutative connected subsets conjugate to $S'$ in $G$, then $H_{i}$ 
appears $k$ times in the summation (however the corresponding affine root systems on $S'$ might be 
different). 
\end{proposition}

\begin{proof}
With Propositions \ref{P:integration2} and \ref{P:density2} available, the proof of this proposition is 
similar to the proof of \cite[Proposition 3.8]{Yu-dimension}.  
\end{proof}

\bigskip 

By Proposition \ref{P:compact-root system}, to study equalities and linear relations among dimension 
data of disconnected subgroups, it suffices to study linear relations among the characters $F_{R,\Gamma^{0}}$ 
for a fixed generalized torus $S$ and affine root systems $R$ on $S$. Write $\Psi=\Psi_{S}$ and $\Psi'=
\Psi'_{S}$. First one can reduce it to the case of $\Psi'$ is an irreducible root system and\footnotemark
$$\bigcap_{\alpha'\in\Psi'}\ker\alpha'=1.$$
\footnotetext{This is analogous to the property of root system for a connected compact semisimple Lie group 
of adjoint type.} 

\begin{question}\label{Q:FR}
Let $S$ be a generalized torus with an inner product on its Lie algebra, and with $\Psi'=\Psi_{S^{0}}$ 
an irreducible root system such that $$\bigcap_{\alpha'\in\Psi'}\ker\alpha'=1.$$ How to find all equalities 
and linear relations among the characters $F_{R,W_{\Psi}}$ for sub-affine root systems $R$ of $\Psi=\Psi_{S}$? 
\end{question}

To solve Question \ref{Q:FR}, we need to first classify affine sub-root systems $R\subset\Psi$. As remarked 
in Subsection \ref{SS:AR}, this consists of three parts: (1), classify sub-root systems $R'$ of $\Psi$; (2),
specify the multiplicity $m_{i}$ for an irreducible factor $R'_{i}$ of $R'$ and multiplicities $m_{1,\alpha'}$ 
and $m_{\alpha'}$ for roots $\alpha'\in R'$; (3), specify liftings in $R$ of simple roots of $R'$. 

In the case of $\Psi'=\BC_{n}$, we should associate polynomials to each irreducible sub-affine root systems 
and study their properties and multiplicative and algebraic relations. These polynomials are variants of the 
polynomials $a_{n}$, $b_{n}$, $c_{n}$, $d_{n}$ in the connected case (cf. \cite[Page 2713]{Yu-dimension}), 
subject to operations on indeterminates and coefficients. It looks to us \cite[Proposition 7.2]{Yu-dimension}) 
has a generalization to the affine case. The cases of $\Psi'=\B_{n}$, $\C_{n}$, $\D_{n}$, $\A_{n-1}$ reduce 
to the $\BC_{n}$ case.
 
In case $\Psi'$ is an exceptional irreducible root system, like in \cite{Yu-dimension} we can calculate the 
dominant terms $2\delta_{R}$ and the polynomials $$f_{R}(t)=\frac{1}{|W_{R}|}\sum_{w\in W_{R}}\epsilon(w)
t^{|\delta_{R}-w\delta_{R}|^{2}}.$$ However, these invariants seem too weak in the affine case. We haven't 
thought seriously yet concerning the case of $\Psi'$ is an exceptional irreducible root system.

\section{Compactness of isospectral set}

A conjecture of Osgood-Philipps-Sarnak states that set of closed Riemannian manifolds with a given 
Laplace spectrum (called an {\it isospectral set} of Riemannian manifolds) should be compact 
(\cite{Osgood-Phillips-Sarnak}, \cite{Gordon}). By considering normal homogeneous spaces, 
we proved a result before. 

\begin{theorem}(\cite[Theorem 3.6]{Yu-compactness})\label{T:spectral}
Given a compact Lie group $G$ equipped with a bi-invariant Riemanniann metric $m_0$, up to isometry, 
any collection of isospectral normal homogeneous spaces of the form $(G/H,m_0)$, where $H\subset G$ 
is a closed subgroup, must be finite. That is, for any closed subgroup $H$ of $G$, up to conjugacy, 
there are finitely many closed subgroups $H_1,\cdots,H_{k}$ of $G$ such that the normal homogeneous 
space $(G/H_{j},m_0)$ is isospectral to $(G/H,m_0)$. 
\end{theorem}

Here we prove a generalization of the above theorem by allowing the metric on $G$ to vary.

\begin{theorem}\label{T:spectral2}
Given a spectrum and a compact semisimple Lie group $G$, there are only finitely many conjugacy 
classes of closed subgroups $H$ of $G$ such that there exists a bi-invariant Riemanniann metric 
$m$ on $G$ which induces a normal homogeneous space $(G/H,m)$ with Laplace spectrum equals to 
the given spectrum.
\end{theorem}

\begin{proof}

First we may assume that $G$ is connected and simply connected. Write $G=G_1\times\cdots G_{s}$ for the 
decomposition of $G$ into simple factors. For each $i$, choose a bi-invariant Riemannian metric $m_{i}$ on 
$G_{i}$. By normalization we may assume that the Laplace operator and Casimir operator coincide on 
$C^{\infty}(G_{i})$ ($1\leq i\leq s$). 

Let $\{(G/H_{n},m_{n}): n\geq 1\}$ be a sequence of normal homogeneous spaces such that the Laplace spectrum 
of each $(G/H_{n},m_{n})$ is equal to a given spectrum, and $H_{n}$ ($n\geq 1$) are non-conjugate to each 
other. Write $$m_{n}=\bigoplus_{1\leq i\leq s}a^{(n)}_{i}m_{i}.$$ By \cite[Theorem 1.1]{Yu-compactness}, 
there exists a closed subgroup $H$ of $G$, a subsequence $\{H_{n_{j}}: j\geq 1\}$ and a sequence $\{g_{j}: 
j\geq 1, g_{j}\in G\}$ such that for all $j\in\mathbb{N}$, \[[H^{0},H^{0}]\subset g_{j}H_{n_{j}}g_{j}^{-1}
\subset H,\] and \[\lim_{j\rightarrow\infty}\mathscr{D}_{H_{n_{j}}}=\mathscr{D}_{H}.\] Substituting 
$\{(G/H_{n},m_{n}): n\geq 1\}$ by a subsequence if necessary we may assume that: for any $n\geq 1$, 
\[[H^{0},H^{0}]\subset H_{n}\subset H,\] and \[\lim_{j\rightarrow\infty}\mathscr{D}_{H_{n}}=
\mathscr{D}_{H}.\] Since $H_{n}$ are assumed to be non-conjugate to each other, at most finitely many of 
them contain $H^{0}$. By removing such exceptions. we may assume that $\dim H_{n}<\dim H$ for all $n$.

We may also assume that each sequence $\{a_{i}^{(n)}: n\geq\}$ converges. Write $$a_{i}=\lim_{n\rightarrow
\infty}a_{i}^{(n)}\in[0,\infty].$$ Without loss of generality we assume that $$a_1=\cdots=a_{u}=0,$$ 
$$0<a_{u+1},\dots,a_{v}<\infty,$$ $$a_{v+1}=\cdots=a_{s}=\infty,$$ where $0\leq u\leq v\leq s$.
Write $$G'=\prod_{u+1\leq i\leq v} G_{i},$$ $$H'=G'\cap H\prod_{1\leq i\leq u}G_{i},$$ and
$$m'=\bigoplus_{u+1\leq i\leq v}a_{i}m_{i}.$$ 

\begin{claim}\label{Cl}
We have $$\prod_{v+1\leq i\leq s} G_{i}\subset H\prod_{1\leq i\leq v} G_{i},$$ and the Laplace 
spectrum of $(G'/H',m')$ equals to the given spectrum.
\end{claim}

Write $$G''=\prod_{1\leq i\leq v} G_{i},\quad G'''=\prod_{v+1\leq i\leq s} G_{i}.$$ Since $$G'''\subset 
HG''$$ by Clail \ref{Cl}, we get $$\dim G/H_{n}>\dim G/H=\dim G''/G''\cap H\geq \dim G'/H'.$$ By 
Claim \ref{Cl}, $(G'/H',m')$ and $(G/H_{n},m_{n})$ are isospectral. As Laplace spectrum determines 
dimension (\cite{Gordon}), we get a contradiction.
\end{proof}

\begin{proof}[Proof of Claim \ref{Cl}] 
Write $\chi_{j}(\rho_{j})$ ($1\leq j\leq s$) for the eigenvalue of Laplace operator associated 
to $(G_{j},m_{j})$ on matrix coefficients of $\rho_{j}\in\widehat{G_{j}}$. We know that: 
$\chi_{j}(\rho_{j})\geq 0$; $\chi_{j}(\rho_{j})=0$ if and only if $\rho_{j}=1$; for any positive 
real number $c$, there are only finitely many $\rho_{j}\in\widehat{G_{j}}$ such that 
$\chi_{j}(\rho_{j})\leq c$. 

Suppose $$\prod_{v+1\leq i\leq s} G_{i}\not\subset H\prod_{1\leq i\leq v} G_{i}.$$ Then, there 
exists a nontrivial irreducible representation $$\rho=\bigotimes_{v+1\leq i\leq s}\rho_{i}$$ 
of $\prod_{v+1\leq i\leq s} G_{i}$ such that $\dim V_{\rho}^{\prod_{v+1\leq i\leq s} G_{i}
\cap H\prod_{1\leq i\leq v}G_{i}}>0$. Take $$0\neq v\in V_{\rho}^{\prod_{v+1\leq i\leq s} G_{i}
\cap H\prod_{1\leq i\leq v}G_{i}}$$ and $0\neq\alpha\in V_{\rho}^{\ast}$. Set 
$$f_{v,\alpha}(g_1,\dots,g_{s})=\alpha((g_{v+1},\dots,g_{s})^{-1}v).$$ Then, $f_{v,\alpha}
\in C^{\infty}(G/H)\subset C^{\infty}(G/H_{n})$ for any $n\geq 1$. The Laplace eigenvalue for 
$f_{v,\alpha}\in C^{\infty}(G/H_{n})$ is equal to $$\sum_{v+1\leq i\leq s}\frac{1}{a_{i}^{(n)}}
\chi_{i}(\rho_{i})>0.$$ When $n\rightarrow\infty$, this value tends to $0$, which is in 
contraduction with the fact that the Laplace spectrum of all $G/H_{n}$ are equal to a given 
spectrum.  

Now we assume $$G'''\subset HG''.$$ Due to $[H^{0},H^{0}]\subset H_{n}$ for any $n\geq 1$, 
we have $G'''\subset H_{n}G''$ for any $n\geq 1$. Each $H_{n}$ is of the form $$H_{n}=
(H_{n}\cap G'')\times\{(\phi_{n}(x),x): x\in G'''\}$$ for some homomorphism $\phi_{n}: 
G'''\rightarrow G''$. Since $G'''$ is semisimple, $$\Hom(G''',G'')/\sim_{G''}$$ 
($\sim_{G''}$ means conjugation action of $G''$ on $\Hom(G''',G'')$ by acting on the 
target of homomorphisms) is a finite set. Replacing $H_{n}$ by a subsequence if necessary, 
we may assume that there exists $\phi\in\Hom(G''',G'')$ such that $\phi_{n}=\phi$ for any 
$n\geq 1$. Moreover we may assume that $$H=(H\cap G'')\times\{(\phi(x),x): x\in G'''\}$$ 
and \[\lim_{n\rightarrow\infty}\mathscr{D}_{H_{n}\cap G''}=\mathscr{D}_{H\cap G''}\] 
(equivalent to say, as dimension data of subgroups of $G''$, or as dimension data of 
subgroups of $G$). Due to $G'''\subset H_{n}G''$, we have $G/H_{n}=G''/G''\cap H_{n}$. 

Let $c$ be a positive real number. Suppose matrix coefficients of $$\rho=\bigotimes_{1\leq i
\leq s}\rho_{i}$$ contributes to Laplace spectrum of $G/H_{n}$ in the eigenvalue scope $[0,c]$. 
Then, $$\sum_{1\leq i\leq s}\frac{1}{a_{i}^{(n)}}\chi_{i}(\rho_{i})\leq c.$$ Due to $a_{i}^{(n)}
\rightarrow a_{i}$ and $a_{i}=0$ for $1\leq i\leq u$, we have: $\rho_{i}$ ($u+1\leq i\leq s$) 
lies in a finite set when $n>>0$; $\rho_{i}=1$ ($1\leq i\leq u$) when $n>>0$. Due to $H_{n}$ 
is of the form $$H_{n}=(H_{n}\cap G'')\times\{(\phi(x),x): x\in G'''\},$$ $\bigotimes_{v+1
\leq i\leq s}\rho_{i}$ is determined by $\bigotimes_{1\leq i\leq v}\rho_{i}$ up to finitely 
many possibility. Thus, each $\rho_{i}$ ($u+1\leq i\leq s$) lies in a finite set when $n>>0$. 
Using $\lim_{n\rightarrow\infty}\mathscr{D}_{G''\cap H_{n}}=\mathscr{D}_{G''\cap H}$, we can 
show that the eigenvalues stabilize and the invariant dimensions $\dim V_{\rho_{i}}^{G'\cap 
H_{n}\prod_{1\leq i\leq u} G_{i}}$ ($u+1\leq i\leq s$) stabilize to $\dim V_{\rho_{i}}^{G'\cap 
H\prod_{1\leq i\leq u} G_{i}}$. This indicates: the Laplace spectrum of $(G'/H',m')$ is larger 
than the given spectrum. Using $H_{n}\subset H$, we can show that the Laplace spectrum of 
$(G'/H',m')$ is smaller than the given spectrum. Therefore, the Laplace spectrum of $(G'/H',m')$ 
equals to the given spectrum. 
\end{proof}

\smallskip 

Motivated by the compactness conjecture of isospectral sets, we think the following statement should hold.

\begin{question}\label{Q:finiteness-normal}
Given a spectrum, can we show that there exist only finitely many normal homogeneous spaces $(G/H,m)$ up 
to isometry with Laplace spectrum equal to the given spectrum? 
\end{question}

In case $G$ and $m$ is given, this follows from Theorem \ref{T:spectral}. 

Any normal homogeneous space is of the form $M=G/H$, where $$G=T\prod_{1\leq i\leq s} G_{i}$$ with $T$ a 
torus and each $G_{i}$ ($1\leq i\leq s$) a connected and simply-connected compact simple Lie group, 
$H\cap T=1$, and $G_{i}\not\subset H$ for any $i$. Let $M=G/H$ be of the this form.

\smallskip 

In case $G$ is semisimple, as $\dim G/H$ is determined from Laplace, we can show that there are only 
finitely many possible $G$ by the condition $G_{i}\not\subset H$ for any $i$. For a fixed $G$, there 
are only finitely many possible $G/H$ by Theorem \ref{T:spectral2}. Then, Question \ref{Q:finiteness-normal} 
reduces the the following statement, which is an algebraic question and seems not hard to prove.   

\begin{question}\label{Q:finiteness-normal2}
Given a spectrum, a connected compact Lie group $G$, and a closed subgroup $H$, can we show that there exist only 
finitely many normal homogeneous spaces $(G/H,m)$ up to isometry with Laplace spectrum equal to the given spectrum?   
\end{question}

\smallskip 

In case $G$ is a torus, $H=1$ by the above assumption. In this case the statement is a theorem of Kneser. 
A simple proof is given in \cite{Wolpert}, which is based on the Mahler compactness theorem for lattices.  

\smallskip 

In case $G$ is neither semisimple nor a torus, again we have only finitely many possible $G$ by dimension 
reason. We can't show yet finitely many possibility of $G/H$ when $G$ is fixed. The difficulty is due to 
there may exist infinitely many metrics on the $T$ part which give isometric metrics on $G/H$. Probably 
a sophisticated use of Mahler compactness theorem could overcome this difficulty.

\section{Questions}\label{S:Question}

Besides some un-solved questions asked in the main body of this paper, Questions \ref{Q:tau-dimension}, \ref{Q:FR}, 
\ref{Q:finiteness-normal}, we have some more questions concerning dimension datum and $\tau$-dimension datum.  

\smallskip 

\begin{question}\label{Q:A=CD}
For subgroups $G,H_1,H_2$ as in Theorem \ref{T:tau-example}, do we have explict formula for $\dim\Hom_{H_{1}}
(\tau_{\lambda},\rho|_{H_{1}})$ ($=\dim\Hom_{H_{2}}(\tau_{\lambda'},\rho|_{H_{2}})$)? 
\end{question}

\smallskip 

\begin{question}\label{Q:asymptotics}
It is hard to calculate the $\tau$-dimension datum in general. But, do we have an asymptotic 
formula for $\dim\Hom_{H}(\tau,\rho|_{H})$?  
\end{question}

We tried to calculate dimension datum or its asymptotics from the character $F_{\Phi,\Gamma^{0}}$, 
but failed. Another way one might think is whether there is a connection of asymptotics of dimension 
datum and the asymptotics studied in \cite{Heckman}. 

\smallskip 

\begin{question}\label{Q:partial data}
Let $G$ be a connected compact Lie group, and $H_1,H_2$ be two closed subgroups. If $$\dim V_{\rho}^{H_1}
=\dim V_{\rho}^{H_2}$$ for all but finitely many $\rho\in\widehat{G}$, is then $\mathscr{D}_{H_1}=
\mathscr{D}_{H_2}$?  
\end{question}

For this question, one may first study the asymptotics of $\dim V_{\rho}^{H_{i}}$ ($i=1,2$) and 
use the asymptotical coincidence to show equality of dimension data. 

\smallskip 

\begin{question}\label{Q:elliptic}
Let $G$ be a connected compact Lie group, and $H_1,H_2$ be two connected closed subgroups. Suppose 
$(Z_{G}(H_{i}))^{0}=Z(G)^{0}$. Does $\mathscr{D}_{H_1}=\mathscr{D}_{H_2}$ imply that $H_1\sim H_2$?   
\end{question}

We call a connected closed subgroup $H$ of a connected compact Lie group $G$ an {\it elliptic subgroup}  
if $(Z_{G}(H_{i}))^{0}=Z(G)^{0}$. Note that, if $\mathscr{D}_{H_1}=\mathscr{D}_{H_2}$ and $H_1$ is and 
elliptic subgroup, then so is $H_2$. One proves this by just taking $\rho$ the adjoint module of $G$. 

Question \ref{Q:elliptic} is related to Questions asked in Arthur's paper \cite{Arthur}. We have made 
some progress along this direction. It remains to be completed. 

\smallskip 

Besides $\tau$-dimension datum and dimension data for disconnected subgroups studied in this paper, in the 
most generality we may study of equalities and linear relations of $\tau$-dimension data for disconnected 
subgroups. In this case the dimension datum is equivalent to a kind of characters which have a mixed form 
of the characters used in studying $\tau$-dimension datum or in studying dimension datum of disconnected 
subgroups. We may develop a theory for this situation in future. 

\smallskip 

Now we know that the dimension datum nearly determines the conjugacy class of a connected closed subgroup, 
but there are a few exceptions. Several mathematicians asked if there are extra information besides dimension 
datum extracted from L-function which together with dimension datum could determine the conjugacy class of the 
$H^{\pi}$ defined by Langlands (\cite{Langlands}). This is an interesting question to think.

\end{document}